\documentclass[a4paper,11pt]{article}
\usepackage[utf8]{inputenc}
\usepackage[english]{babel}

\usepackage{amsmath}
\usepackage{amssymb}
\usepackage{amsthm}
\usepackage{enumerate}
\usepackage{fullpage}
\usepackage{mathpazo} 
\usepackage{mathtools}
\usepackage[natbibapa]{apacite} 
\usepackage{tikz}
\usetikzlibrary{calc,math,shapes.geometric,arrows}
\usepackage[normalem]{ulem}
\usepackage{hyperref}
\usepackage{cleveref}

\newtheorem{theorem}{Theorem}[section]

\newtheorem{lemma}[theorem]{Lemma}
\newtheorem{corollary}[theorem]{Corollary}
\theoremstyle{definition}
\newtheorem{example}[theorem]{Example}
\newtheorem{definition}[theorem]{Definition}

\newtheorem{remark}[theorem]{Remark}

\theoremstyle{remark}
\newtheorem*{claim}{Claim}

\newcommand{\conv}{\operatorname{conv}}

\DeclareRobustCommand{\orcidlink}[1]{%
  \href{https://orcid.org/#1}{\begingroup\normalfont%
    \raisebox{-\fontchardp\font`q}{%
      \includegraphics[height=\fontcharht\font`/+\fontchardp\font`q]{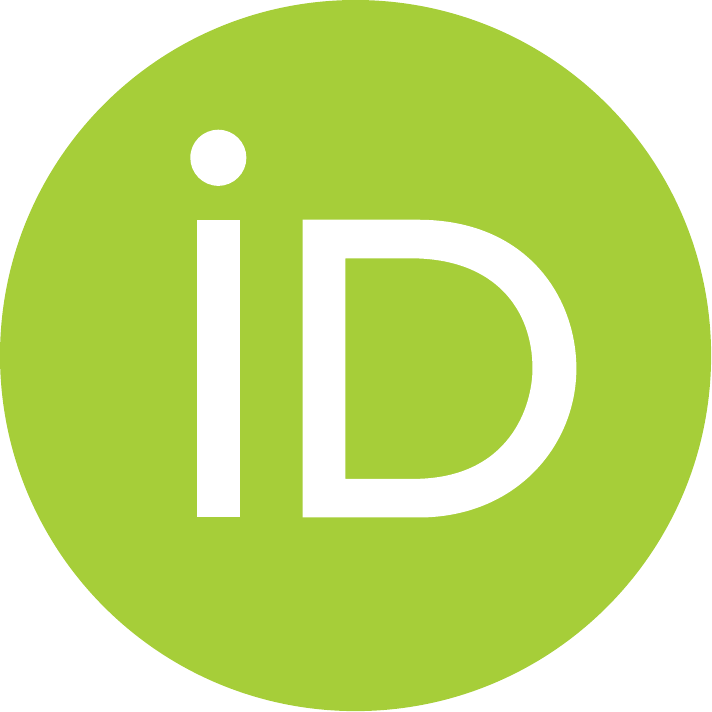}%
    }\endgroup%
    ~#1%
  }%
}
\newcommand{\orcid}[1]{\footnote{\orcidlink{#1}}}

\pgfdeclarelayer{bg} 
\pgfsetlayers{bg,main}

\title{On the Split Closure of the Periodic Timetabling Polytope}
\author{%
\begin{tabular}{cc}
Niels Lindner\orcid{0000-0002-8337-4387} & Berenike Masing\orcid{0000-0001-7201-2412}\\
Freie Universität Berlin & Zuse Institute Berlin\\
\texttt{lindner@zib.de} & \texttt{masing@zib.de}
\end{tabular}%
}
\date{\today}

\begin{document}

\maketitle

\begin{abstract}
The Periodic Event Scheduling Problem (PESP) is the central mathematical tool for periodic timetable optimization in public transport. PESP can be formulated in several ways as a mixed-integer linear program with typically general integer variables.
We investigate the split closure of these formulations and show that split inequalities are identical with the recently introduced flip inequalities. While split inequalities are a general mixed-integer programming technique, flip inequalities are defined in purely combinatorial terms, namely cycles and arc sets of the digraph underlying the PESP instance. It is known that flip inequalities can be separated in pseudo-polynomial time. We prove that this is best possible unless P $=$ NP, but also observe that the complexity becomes linear-time if the cycle defining the flip inequality is fixed.
Moreover, introducing mixed-integer-compatible maps, we compare the split closures of different formulations, and show that reformulation or binarization by subdivision do not lead to stronger split closures. Finally, we estimate computationally how much of the optimality gap of the instances of the benchmark library PESPlib can be closed exclusively by split cuts, and provide better dual bounds for five instances.
\end{abstract}


\paragraph{Mathematics Subject Classification (MSC2020):} 90C11, 90C35, 90B35, 90B20

\section{Introduction}
\label{sec:intro}
The timetable is the core of a public transportation system. It serves as a basis for cost-sensitive tasks such as vehicle and crew scheduling, and is required for accurate planning of passenger routes. A high-quality timetable is thus of utmost importance for a well-planned transportation system. Particularly in the context of urban traffic, a large number of transportation networks are operated with a periodic pattern, creating the demand to optimize periodic timetables.  
The standard mathematical model for this task is the \emph{Periodic Event Scheduling Problem} (PESP) introduced by \cite{serafini_mathematical_1989}. PESP is a combinatorial optimization problem on a digraph with respect to a certain period time, and it is notoriously hard: Deciding whether a feasible periodic timetable exists is NP-complete for any fixed period time $T \geq 3$ \citep{odijk_construction_1994}. The feasibility problem remains NP-hard on graphs with bounded treewidth \citep{lindner_parameterized_2022}. The difficulty of PESP is also reflected in the fact that since its establishment in 2012, none of the instances of the benchmark library PESPlib could be solved to proven optimality up to date \citep{PESPlib}. Nevertheless, many primal heuristics have been developed \citep{nachtigall_solving_2008,grosmann_solving_2012,patzold_matching_2016,borndorfer_concurrent_2020,lindner_timetable_2022,goerigk_improved_2017}, and there are success stories concerning the implementation of mathematically optimized timetables in practice \citep{liebchen_first_2008,kroon_new_2009}.

PESP can be formulated as a mixed-integer linear program (MIP) in a multitude of ways \citep{liebchen_periodic_2006}. Several studies of the \emph{periodic timetabling polytope} have been conducted, leading to the discovery of families of cutting planes, such as, e.g., \emph{cycle inequalities} \citep{odijk_construction_1994}, \emph{change-cycle inequalities} \citep{nachtigall_cutting_1996}, and more recently, \emph{flip inequalities}\citep{lindner_determining_2020}. The separation of cycle and change-cycle inequalities is known to be NP-hard \citep{borndorfer_separation_2020}, and flip inequalities are a superset of both cycle and change-cycle inequalities \citep{lindner_determining_2020}.
A common theme that cycle, change-cycle and flip inequalities share as well with other families of cutting planes \citep{nachtigall_cutting_1996,lindner_train_2000} is that they are all described in purely combinatorial terms. For example, flip inequalities are determined by a cycle and a set of arcs of the underlying digraph of the PESP instance. 

In this paper, we pursue a somewhat opposite strategy: Rather than starting with a combinatorial analysis, we investigate \emph{split inequalities}, a general-purpose tool for treating MIPs introduced by \cite{cook_chvatal_1990} as an analogon to the Chvátal closure for pure integer programs. The \emph{split closure} given by these inequalities has several nice properties: It is a polyhedron \citep{cook_chvatal_1990,conforti_polyhedral_2010}, coincides with the closure given by mixed-integer rounding and Gomory mixed-integer cuts \citep{nemhauser_recursive_1990,cornuejols_elementary_2001}, and leads to finite cutting plane algorithms for binary MIPs \citep{balas_lift-and-project_1993}. 

While the second Chvátal closure for a pure IP formulation has already been investigated by \cite{liebchen_second_2008}, we apply split closure techniques to proper mixed-integer formulations of PESP. Our first result is the following correspondence (\Cref{thm:split-is-flip}): Every non-trivial split inequality is a non-trivial flip inequality, and vice versa. The split closure of the periodic timetabling polytope is therefore identical with the closure given by the flip inequalities. Moreover, the split inequalities coming from split disjunctions, where one of the two sides of the split is empty, coincide with Odijk's cycle inequalities (\Cref{thm:oneside}).

In general, the separation of split inequalities is NP-hard \citep{caprara_separation_2003}. In the periodic timetabling situation, we show in \Cref{thm:sepa-hardness} that it is weakly NP-hard to separate maximally violated split/flip inequalities. This is best possible unless P $=$ NP, as \cite{lindner_determining_2020} have already outlined a pseudo-polynomial-time algorithm. The separation problem can however by solved by a parametric IP in the spirit of \cite{balas_optimizing_2008} and \cite{bonami_optimizing_2012}, which in the special case of PESP boils down to a sequence of $\lfloor T/2 \rfloor - 1$ standard IPs (\Cref{thm:sepa-parametric-mip}). In the event that the cycle defining a flip inequality is fixed, the separation becomes linear-time (\Cref{thm:sepa-fixed-cycle}).

So far, the results on the split closure of the periodic timetabling polytope apply for the cycle-based MIP formulation of PESP \citep{nachtigall_cutting_1996,liebchen_integral_2009}. Another popular formulation is the incidence-based formulation that is straightforward from the original problem definition by \cite{serafini_mathematical_1989}. In order to compare the split closures of two different MIP formulations, we introduce \emph{mixed-integer-compatible maps}, i.e., affine maps that map mixed-integer points to mixed-integer points. These maps have the general property that they map split closures into split closures (\Cref{thm:split-descendance}). For PESP, the polytope defined by the cycle-based formulation turns out to be a mixed-integer-compatible projection of the polytope defined by the incidence-based formulation. However, we show that the restriction of this projection to split closures is surjective, so that there is no gain in information concerning split cuts when switching to a different formulation (\Cref{thm:free-augmentation}). More results that can be proven using mixed-integer-compatible maps are the following: The split closure commutes with Cartesian products (\Cref{thm:split-product}). This enables us to show that the split closure of PESP instances on cactus graphs is exact (\Cref{thm:split-block}).

The behavior of split or lift-and-project closures with respect to binarizations, i.e., MIP reformulations with only binary integer variables, have received some attention lately \citep{dash_binary_2018,aprile_binary_2021}. In the context of PESP, the incidence-based MIP formulation can be binarized in a combinatorial manner by subdivision of arcs. Although split closures of binary MIPs are known to be much better behaved \citep{balas_lift-and-project_1993}, we prove by another application of mixed-integer-compatible maps that this binarization procedure does also not lead to stronger split closures (\Cref{thm:binarization}).

Finally, we evaluate split closures in practice. To this end, we consider the 22 PESPlib instances and some derived subinstances. We devise an algorithmic procedure to optimize over the split closure making use of our theoretical insights. Our separation algorithm consists of a heuristic and an exact part. The outcome is that although the split closure closes a significant part of the primal-dual gap, it is almost never exact. However, our separation method produces incumbent dual bounds for 5 of the PESPlib instances.

The paper is structured as follows: We summarize the relevant definitions and notions for PESP in \Cref{sec:pesp}. The correspondence between split and flip inequalities follows in \Cref{sec:split}. The subsequent \Cref{sec:sepa} is devoted to separation of split/flip inequalities. Mixed-integer compatible maps and the results on comparing split closures of different formulations are presented in \Cref{sec:comparing}. Our computational results can be found in \Cref{sec:experiments}. We conclude the paper in \Cref{sec:conclusion}.

\section{Periodic Event Scheduling}
\label{sec:pesp}
The Periodic Event Scheduling Problem has originally been introduced by \cite{serafini_mathematical_1989}, and has gained much attention ever since. In this chapter, we establish the basics, formally state the problem, introduce two equivalent model formulations and introduce our main object of interest, the periodic timetabling polytope. 

\subsection{Problem Definition}
\label{ss:problem-definition}

An instance of the \emph{Periodic Event Scheduling Problem} (PESP) is given by a 5-tuple $(G, T, \ell, u, w)$, where
\begin{itemize}
    \item $G = (V, A)$ is a directed graph,
    \item $T \in \mathbb N$, $T \geq 2$, is a \emph{period time},
    \item $\ell \in \mathbb Z^A$ is a vector of \emph{lower bounds},
    \item $u \in \mathbb Z^A$ is a vector of \emph{upper bounds},
    \item $w \in \mathbb R^A$ is a vector of \emph{weights}.
\end{itemize}
A \emph{periodic tension} is a vector $x \in \mathbb R^A$ with $\ell \leq x \leq u$ such that 
\begin{equation}
    \label{eq:pix}
    \exists\, \pi \in [0, T)^V: \quad \forall a = (i, j) \in A:\quad  x_a \equiv \pi_j - \pi_i  \mod T.
\end{equation}

In this case, the vector $\pi$ is called a \emph{periodic timetable}. In the context of periodic timetabling in public transport, the vertices of $G$ typically correspond to arrival or departure \emph{events} of vehicles at some station. The arcs of $G$ are \emph{activities}; they model relations between the events such as, e.g., driving between two stations, dwelling at a station, or passenger transfers \citep{liebchen_modeling_2007}. A periodic timetable $\pi$ thus assigns timings in $[0, T)$ to each event, repeating periodically with period $T$. The periodic tension $x$ collects the activity durations, which are supposed to lie within the feasible interval $[\ell, u]$. A typical source for the weight of an arc is the estimated number of passengers using the corresponding activity. A reasonable quality indicator of a periodic timetable is hence $w^\top x$, the total travel time of all passengers.

\begin{definition}[\citealt{serafini_mathematical_1989}]\label{def:pesp}
Given $(G, T, \ell, u, w)$ as above, the \emph{Periodic Event Scheduling Problem} is to find a periodic tension $x$ such that $w^\top x$ is minimum, or to decide that none exists.
\end{definition}

\begin{example}
\Cref{fig:pesp-example} shows a small PESP instance together with an optimal periodic tension and a compatible periodic timetable.

\begin{figure}[htbp]
    \centering
    \def\circledarrow#1#2#3{ 
	\draw[#1,<-] (#2) +(60:#3) arc(60:-260:#3);
}

\begin{tikzpicture}[scale=1]
	\definecolor{highlight}{HTML}{4eed98}
	\definecolor{mint}{HTML}{14b861}
	\definecolor{darkblue}{HTML}{003399}
	\tikzstyle{a} = [line width=1.3, ->]
	\tikzstyle{am} = [line width=6, highlight]
	\tikzstyle{t} = [midway, font=\footnotesize]
	\tikzstyle{ta} = [t, above]
	\tikzstyle{tb} = [t, below, blue]
	\tikzstyle{tr} = [t, right]
	\tikzstyle{tl} = [t, left]
	\tikzstyle{v} = [draw, circle, inner sep=4, fill = gray!20, outer sep = 2]
	\tikzstyle{c} = [gray]
	\tikzmath{\h = 4; \v = 2;};
	\node[v] (A1) at (0,0) {0};
	\node[v] (A2) at ($(1*\h,0)$) {1};
	\node[v] (A3) at ($(2*\h,0)$) {4};
	\node[v] (A4) at ($(3*\h,0)$) {5};
	
	\node[v] (B1) at ($(0*\h,\v)$) {9};
	\node[v] (B2) at ($(1*\h,\v)$) {8};
	\node[v] (B3) at ($(2*\h,\v)$) {5};
	\node[v] (B4) at ($(3*\h,\v)$) {4};
	
	\draw[a] (A1) -- node[ta] {$[1,2],11$}  node[tb] {$1$} (A2);
	\draw[a] (A2) -- node[ta] {$[3,6],11$} node[tb] {$3$} (A3);
	\draw[a] (A3) -- node[ta] {$[1,2],11$} node[tb] {$1$} (A4);
	
	\draw[a] (B4) -- node[ta] {$[1,2],11$} node[tb] {$1$} (B3);
	\draw[a] (B3) -- node[ta] {$[3,6],11$} node[tb] {$3$} (B2);
	\draw[a] (B2) -- node[ta] {$[1,2],11$} node[tb] {$1$} (B1);
	
	\draw[a] (A4) to[bend right] node[ta, rotate = -90] {$[1,10],10$} node[tb, rotate = -90] {$9$} (B4);
	\draw[a] (B1)  to[bend right] node[ta, rotate=90] {$[1,10],10$} node[tb, rotate = 90] {$1$} (A1);
	\draw[a] (B2) -- node[ta, rotate = 90] {$[3,12],0$}  node[tb, rotate = 90] {$3$}  (A2);
	\draw[a] (A3) -- node[ta, rotate = -90] {$[3,12],0$} node[tb, rotate = -90] {$11$} (B3);

\end{tikzpicture}
    \caption{A PESP instance on a digraph $G = (V, A)$ with $T=10$. The upper label of an arc $a \in A$ is $[\ell_a, u_a], w_a$. The blue lower arc labels indicate a periodic tension $x$ compatible with the periodic timetable $\pi$ as given by the vertex labels.}
    \label{fig:pesp-example}
\end{figure}

\end{example}

\begin{remark}
\label{rem:preprocessing}
As described, e.g., by \cite{liebchen_periodic_2006}, any PESP instance can be preprocessed in such a way that $G$ contains no loops and is weakly connected, $0 \leq \ell < T$ and $\ell \leq u < \ell + T$.
\end{remark}


\subsection{Mixed-Integer Programming Formulations}
\label{subsec:mip}

PESP can be formulated as a mixed-integer linear program in several ways \citep{liebchen_periodic_2006}. The \emph{incidence-based} model is a straightforward interpretation of the problem definition, introducing auxiliary integer \emph{periodic offsets} to resolve the modulo constraints \eqref{eq:pix}:

\begin{equation}
\label{eq:mip-incidence}
    \begin{aligned}
    & \text{Minimize} & w^\top x \\
    & \text{s.t.} & x_a &= \pi_j - \pi_i + T p_a &\quad a = (i,j) \in A\\
    & & \ell_{a} &\leq x_{a} \leq u_{a}, &\quad a \in A\\
    & & 0 &\leq \pi_i \leq T-1, &\quad i \in V,\\
    & & p_{a} &\;\text{integer}, &\quad a \in A.
    \end{aligned}
\end{equation}

When all periodic offsets $p_a$ are fixed, \eqref{eq:mip-incidence} becomes a linear program with a totally unimodular constraint matrix. It is hence no restriction to assume that $x$ and $\pi$ are integral, so that the bound $\pi < T$ in \eqref{eq:pix} can safely be replaced with $\pi \leq T-1$. For the purpose of this paper, we will however not treat \eqref{eq:mip-incidence} as a pure integer program, as was done by \cite{liebchen_second_2008}. We will instead investigate proper mixed-integer formulations, where the periodic tension variables $x$ and the periodic timetable variables $\pi$ are considered as continuous variables.

An alternative MIP formulation for PESP is the \emph{cycle-based} formulation, which has been reported to be computationally beneficial (see., e.g., \citealt{peeters_cyclic_2003,liebchen_first_2008,liebchen_integral_2009,borndorfer_concurrent_2020,schiewe_periodic_2020}):

\begin{equation}
\label{eq:mip-cycle}
    \begin{aligned}
    & \text{Minimize} & w^\top x \\
    & \text{s.t.} & \Gamma x &= T z,\\
    & & \ell &\leq x\leq u,\\
    & & z &\;\text{integer}.
    \end{aligned}
\end{equation}

In \eqref{eq:mip-cycle}, $x$ represents a periodic tension, and $z$ is an integral \emph{cycle offset}. A periodic timetable $\pi$ can be recovered from $x$ by a graph traversal.

To explain the further ingredients of the formulation \eqref{eq:mip-cycle}, we will require more definitions about cycles, cycle spaces and cycle bases, see \cite{kavitha_cycle_2009} for an overview. The \emph{cycle space} $\mathcal C$ of $G$ is the abelian group
$$\mathcal C \coloneqq \left\{ \gamma \in \mathbb Z^A \,\middle|\, \forall i \in V: \sum_{a \in \delta^+(i)} \gamma_a = \sum_{a \in \delta^-(i)} \gamma_a \right\}.$$
In terms of linear algebra, $\mathcal C$ is the kernel over the integers of the incidence matrix of $G$; in the language of network flows, $\mathcal C$ is the space of all integer-valued (and arbitrarily signed) circulations in $G$. The rank of $\mathcal C$ is the \emph{cyclomatic number} $\mu$ of $G$. We assume that $G$ is weakly connected (\Cref{rem:preprocessing}), so that $\mu = |A| - |V| + 1$.

A vector $\gamma \in \mathcal C \cap \{-1,0,1\}^A$ will be called an \emph{oriented cycle}. When ignoring arc directions, the support $\{a \in A \mid \gamma_a \neq 0\}$ makes up a possibly non-simple cycle in $G$. We call arcs $a$ with $\gamma_a > 0$ \emph{forward} and those with $\gamma_a < 0$ \emph{backward}. Any $\gamma \in \mathcal C$ can be decomposed into its positive resp.\ negative part $\gamma_+$ resp.\ $\gamma_-$, i.e., $\gamma_+ \coloneqq \max(\gamma, 0)$ and $\gamma_- \coloneqq \max(-\gamma, 0)$. The \emph{length} of an oriented cycle $\gamma$ is $|\gamma| \coloneqq |\{a \in A \mid \gamma_a \neq 0\}|$.

A set $B$ of $\mu$ oriented cycles is called an \emph{integral cycle basis} of $G$ if $B$ is a basis for $\mathcal C$ as an abelian group, i.e., if every element of the cycle space $\mathcal C$ can be written as a unique integral linear combination of the oriented cycles in $B$. A particular class of integral cycle bases are the \emph{(strictly) fundamental cycle bases}: Let $\mathcal T$ be some spanning tree of $G$. Then the fundamental cycle induced by the co-tree arc $a$ of $\mathcal T$ is the unique cycle $\gamma$ obtained by adding $a$ to $\mathcal T$ with the convention that $\gamma_a = 1$. A fundamental cycle basis is then given by the collection of $\mu$ fundamental cycles of $\mathcal T$. 
Arranging the oriented cycles of an integral cycle basis $B$ as rows of a matrix, we obtain a \emph{cycle matrix} $\Gamma \in \{-1,0,1\}^{B \times A}$.

\begin{example}
In the example from \Cref{fig:pesp-example}, we have $\mu = 3$. An integral cycle basis $B$ is outlined in \Cref{fig:pesp-cycle-basis}.

\begin{figure}[htbp]
    \centering
    \def\circledarrow#1#2#3{ 
	\draw[#1,<-] (#2) +(60:#3) arc(60:-260:#3);
}

\begin{tikzpicture}[scale=1]
	\definecolor{highlight}{HTML}{4eed98}
	\definecolor{mint}{HTML}{14b861}
	\definecolor{darkblue}{HTML}{003399}
	\tikzstyle{a} = [line width=1.3, ->]
	\tikzstyle{am} = [line width=6, highlight]
	\tikzstyle{t} = [midway, font=\footnotesize]
	\tikzstyle{ta} = [t, above]
	\tikzstyle{tb} = [t, below, blue]
	\tikzstyle{tr} = [t, right]
	\tikzstyle{tl} = [t, left]
	\tikzstyle{v} = [draw, circle, inner sep=4, fill = gray!20, outer sep = 2]
	\tikzstyle{c} = [gray]
	\tikzmath{\h = 4; \v = 2;};
	\node[v] (A1) at (0,0) {0};
	\node[v] (A2) at ($(1*\h,0)$) {1};
	\node[v] (A3) at ($(2*\h,0)$) {4};
	\node[v] (A4) at ($(3*\h,0)$) {5};
	
	\node[v] (B1) at ($(0*\h,\v)$) {9};
	\node[v] (B2) at ($(1*\h,\v)$) {8};
	\node[v] (B3) at ($(2*\h,\v)$) {5};
	\node[v] (B4) at ($(3*\h,\v)$) {4};
	
	\draw[a] (A1) -- node[ta] {$[1,2],11$}  node[tb] {$1$} (A2);
	\draw[a] (A2) -- node[ta] {$[3,6],11$} node[tb] {$3$} (A3);
	\draw[a] (A3) -- node[ta] {$[1,2],11$} node[tb] {$1$} (A4);
	
	\draw[a] (B4) -- node[ta] {$[1,2],11$} node[tb] {$1$} (B3);
	\draw[a] (B3) -- node[ta] {$[3,6],11$} node[tb] {$3$} (B2);
	\draw[a] (B2) -- node[ta] {$[1,2],11$} node[tb] {$1$} (B1);
	
	\draw[a] (A4) to[bend right] node[ta, rotate = -90] {$[1,10],10$} node[tb, rotate = -90] {$9$} (B4);
	\draw[a] (B1)  to[bend right] node[ta, rotate=90] {$[1,10],10$} node[tb, rotate = 90] {$1$} (A1);
	\draw[a] (B2) -- node[ta, rotate = 90] {$[3,12],0$}  node[tb, rotate = 90] {$3$}  (A2);
	\draw[a] (A3) -- node[ta, rotate = -90] {$[3,12],0$} node[tb, rotate = -90] {$11$} (B3);
	
	\node[c] (gamma1) at ($(0.5*\h,0.5*\v)$) {$\gamma_1$};
	\circledarrow{ gray}{gamma1}{0.5cm};
	\node[c] (gamma2) at ($(1.5*\h,0.5*\v)$) {$\gamma_2$};
	\circledarrow{ gray}{gamma2}{0.5cm};
	\node[c] (gamma3) at ($(2.5*\h,0.5*\v)$) {$\gamma_3$};
	\circledarrow{ gray}{gamma3}{0.5cm};
	\begin{pgfonlayer}{bg}
	
	\draw[am] (A1) --  (A2);
	\draw[am] (A2) -- (A3);
	\draw[am] (A3) --  (A4);
	\draw[am] (A2) --  (B2);
	\draw[am] (A3) --  (B3);
	\draw[am] (A4) to[bend right]  (B4);
	\draw[am] (B1)  to[bend right] (A1);
	
\end{pgfonlayer}
	
\end{tikzpicture}
    \caption{In the instance from \Cref{fig:pesp-example}, the oriented cycles $\gamma_1, \gamma_2, \gamma_3$ constitute an integral cycle basis, as they are the fundamental cycles of the highlighted spanning tree. The cycle $\gamma_2$ uses only forward arcs, while $\gamma_1$ and $\gamma_3$ have both forward and backward arcs. The tension $\gamma_3^\top x$ along $\gamma_3$ is $1 + 1 + 9 - 1 = 10 \equiv 0 \bmod 10$, and $\gamma_1^\top x$ and $\gamma_2^\top x$ are integer multiples of $T = 10$ as well.}
    \label{fig:pesp-cycle-basis}
\end{figure}
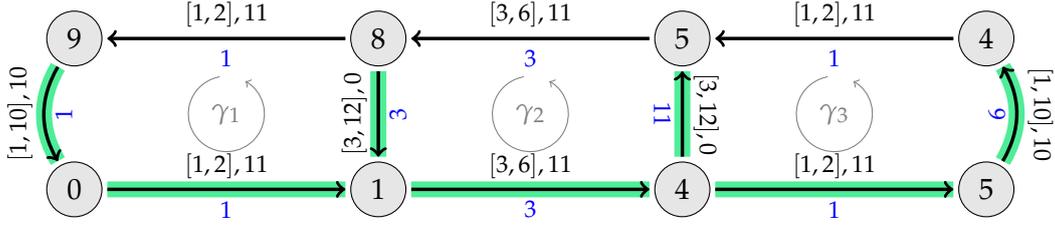

\end{example}

The following theorem shows that the MIP \eqref{eq:mip-cycle} is indeed a valid formulation of PESP.

\begin{theorem}[Cycle periodicity property, \citealp{liebchen_integral_2009}]
\label{thm:cycle-periodicity}
For a vector $x \in \mathbb R^A$, the following are equivalent:
\begin{enumerate}[\normalfont (a)]
    \item $x$ satisfies condition \eqref{eq:pix},
    \item $\gamma^\top x \equiv 0 \bmod T$ for all $\gamma \in \mathcal C$,
    \item $\Gamma x \equiv 0 \bmod T$ for the cycle matrix $\Gamma$ of an integral cycle basis of $G$.
\end{enumerate}
\end{theorem}

In the sequel, we will focus on the cycle-based formulation \eqref{eq:mip-cycle}, which is justified by the following remark.

\begin{remark}
    \label{rem:incidence-is-cycle}
    The incidence-based formulation \eqref{eq:mip-incidence} is a particular incarnation of the cycle-based formulation \eqref{eq:mip-cycle} in the following sense:
    Let $I = (G, T, \ell, u, w)$ be a PESP instance. We can augment $I$ to an instance $I'$ such that the incidence-based MIP formulation \eqref{eq:mip-incidence} for $I$ coincides with the cycle-based MIP formulation \eqref{eq:mip-cycle} for $I'$ for a certain integral cycle basis $B$ with cycle matrix $\Gamma$. To this end, we add a new vertex $s$ and connect it to every original vertex $i \in V$. Set $\ell_{si} \coloneqq 0, u_{si} \coloneqq T-1, w_{si} \coloneqq 0$. The subgraph $\mathcal T$ on the arcs $\{(s,i) \mid i \in V\}$ is a spanning tree of the augmented graph. Each fundamental cycle has the vertex sequence $(s, i, j, s)$ for some arc $a = (i,j) \in A$; we assume that the arcs $(s,i)$ and $(i,j)$ are forward, and that the arc $(s,j)$ is backward. The constraint in \eqref{eq:mip-cycle} for the cycle $(s, i, j, s)$ is then given by $x_{si} + x_{a} - x_{sj} = T z_{a}$. Relabeling $x_{si}$ as $\pi_i$ for $i \in V$ and $z_a$ as $p_a$ for $a \in A$, the formulation \eqref{eq:mip-cycle} for the augmented instance and the cycle matrix $\Gamma$ given by the fundamental cycle basis with respect to $\mathcal T$ indeed turns out to be the same as the formulation \eqref{eq:mip-incidence} for the original instance $I$. In particular, the PESP instances $I$ and $I'$ can be considered equivalent.
\end{remark}

\begin{example}
\Cref{fig:pesp-augmentation} shows the augmented instance $I'$ obtained from the instance $I$ from \Cref{fig:pesp-example} according to \Cref{rem:incidence-is-cycle}.

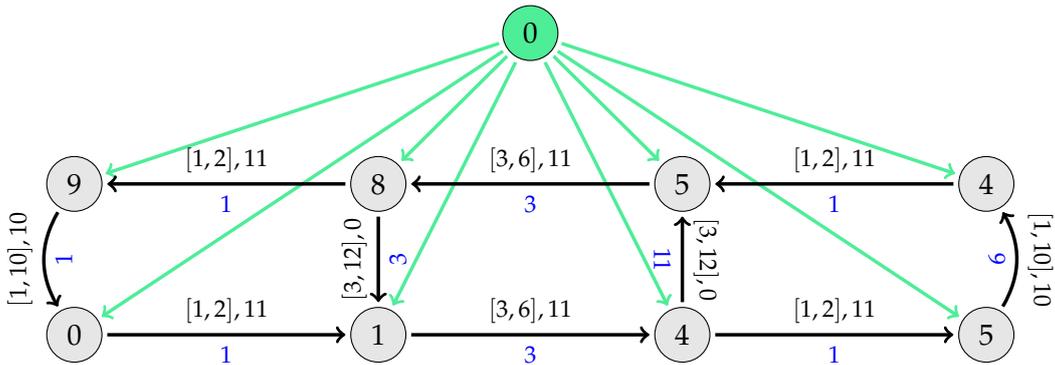
\begin{figure}[htbp]
    \centering
    \def\circledarrow#1#2#3{ 
	\draw[#1,<-] (#2) +(60:#3) arc(60:-260:#3);
}

\begin{tikzpicture}[scale=1]
	\definecolor{highlight}{HTML}{4eed98}
	\definecolor{mint}{HTML}{14b861}
	\definecolor{darkblue}{HTML}{003399}
	\tikzstyle{a} = [line width=1.3, ->]
	\tikzstyle{ag} = [line width=1.3, ->, highlight]
	\tikzstyle{am} = [line width=6, highlight]
	\tikzstyle{t} = [midway, font=\footnotesize]
	\tikzstyle{ta} = [t, above]
	\tikzstyle{tb} = [t, below, blue]
	\tikzstyle{tr} = [t, right]
	\tikzstyle{tl} = [t, left]
	\tikzstyle{v} = [draw, circle, inner sep=4, fill = gray!20, outer sep = 2]
	\tikzstyle{c} = [gray]
	\tikzmath{\h = 4; \v = 2;};
	
	\node[v, fill=highlight] (S) at ($(1.5*\h,2*\v)$) {0};
	
	\node[v] (A1) at (0,0) {0};
	\node[v] (A2) at ($(1*\h,0)$) {1};
	\node[v] (A3) at ($(2*\h,0)$) {4};
	\node[v] (A4) at ($(3*\h,0)$) {5};
	
	\node[v] (B1) at ($(0*\h,\v)$) {9};
	\node[v] (B2) at ($(1*\h,\v)$) {8};
	\node[v] (B3) at ($(2*\h,\v)$) {5};
	\node[v] (B4) at ($(3*\h,\v)$) {4};
	
	\draw[a] (A1) -- node[ta] {$[1,2],11$}  node[tb] {$1$} (A2);
	\draw[a] (A2) -- node[ta] {$[3,6],11$} node[tb] {$3$} (A3);
	\draw[a] (A3) -- node[ta] {$[1,2],11$} node[tb] {$1$} (A4);
	
	\draw[a] (B4) -- node[ta] {$[1,2],11$} node[tb] {$1$} (B3);
	\draw[a] (B3) -- node[ta] {$[3,6],11$} node[tb] {$3$} (B2);
	\draw[a] (B2) -- node[ta] {$[1,2],11$} node[tb] {$1$} (B1);
	
	\draw[a] (A4) to[bend right] node[ta, rotate = -90] {$[1,10],10$} node[tb, rotate = -90] {$9$} (B4);
	\draw[a] (B1)  to[bend right] node[ta, rotate=90] {$[1,10],10$} node[tb, rotate = 90] {$1$} (A1);
	\draw[a] (B2) -- node[ta, rotate = 90] {$[3,12],0$}  node[tb, rotate = 90] {$3$}  (A2);
	\draw[a] (A3) -- node[ta, rotate = -90] {$[3,12],0$} node[tb, rotate = -90] {$11$} (B3);
	
	\begin{pgfonlayer}{bg}
	\draw[ag] (S) -- (A1);
	\draw[ag] (S) -- (A2);
	\draw[ag] (S) -- (A3);
	\draw[ag] (S) -- (A4);
	\draw[ag] (S) -- (B1);
	\draw[ag] (S) -- (B2);
	\draw[ag] (S) -- (B3);
	\draw[ag] (S) -- (B4);
	\end{pgfonlayer}
	
	
\end{tikzpicture}
    \caption{Augmentation of the instance in \Cref{fig:pesp-example} according to \Cref{rem:incidence-is-cycle}. The new vertex $s$ and the new arcs $(s, i)$ are highlighted in green. The highlighted arcs form a spanning tree of the augmented instance. The periodic tension $x_{si}$ of a highlighted arc $(s, i)$ can be read off the timetable value $\pi_i$ given as vertex label at the gray vertex $i$.}
    \label{fig:pesp-augmentation}
\end{figure}

\end{example}

\subsection{The Periodic Timetabling Polytope}
\label{ss:timetabling-polytope}
Before analyzing the split closure, we need to understand the geometric object behind the feasible region of a PESP instance, and also of its natural LP relaxation.

\begin{definition}
\label{def:polytopes}
For a PESP instance $(G, T, \ell, u, w)$ and a cycle matrix $\Gamma$ of an integral cycle basis $B$, define
\begin{align*}
    \mathcal P &\coloneqq \{ (x, z) \in \mathbb R^A \times \mathbb R^B \mid \Gamma x = Tz, \ell \leq x \leq u \},\\
    \mathcal P_\mathrm{I} &\coloneqq \conv \{ (x, z) \in \mathbb R^A \times \mathbb Z^B \mid \Gamma x = Tz, \ell \leq x \leq u \}.
\end{align*}
We will call $\mathcal P$ the \emph{fractional periodic timetabling polytope} and  $\mathcal P_{I}$ the \emph{integer periodic timetabling polytope}.
\end{definition}

$\mathcal P_\mathrm{I}$ is the convex hull of the feasible solutions to \eqref{eq:mip-cycle}, and the fractional periodic timetabling polytope $\mathcal P$ is the polyhedron associated to the natural linear programming relaxation of \eqref{eq:mip-cycle}. 
Observe that this relaxation is very weak: $\mathcal P$ is combinatorially equivalent to the hyperrectangle $\prod_{a \in A} [\ell_a, u_a]$, and an optimal vertex of the LP relaxation of \eqref{eq:mip-cycle} is given by $(\ell, \Gamma \ell/T)$.

\begin{remark}
\label{rem:basis-indep}
The choice of a cycle basis $\Gamma$ is not essential for the definition of $\mathcal P$ and $\mathcal P_\mathrm{I}$: If $\Gamma'$ is the cycle matrix of another integral cycle basis, then there is a unimodular matrix $U$ such that $\Gamma' = U \Gamma$, and $(x, z) \mapsto (x, Uz)$ is a $\mathbb Z$-linear isomorphism.
\end{remark}

Several classes of valid inequalities for $\mathcal P_\mathrm{I}$ are known \citep{odijk_construction_1994,nachtigall_cutting_1996,nachtigall_periodic_1998,lindner_train_2000,lindner_determining_2020}. We will focus on those that are defined in terms of elements of the cycle space $\mathcal C$. The cycle periodicity property (\Cref{thm:cycle-periodicity}) immediately shows:
\begin{theorem}[\citealp{odijk_construction_1994}]
\label{thm:odijk}
    Let $\gamma \in \mathcal C$. Then the following \emph{cycle inequality} holds for all $(x, z) \in \mathcal P_\mathrm{I}$:
    \begin{equation}
    \label{eq:odijk}
    \left\lceil \frac{\gamma_+^\top \ell - \gamma_-^\top u}{T} \right\rceil \leq \frac{\gamma^\top x}{T} \leq \left\lfloor \frac{\gamma_+^\top u - \gamma_-^\top \ell}{T} \right\rfloor.
    \end{equation}
\end{theorem}
Since the rows of the cycle matrix $\Gamma$ are oriented cycles, \Cref{thm:odijk} implies bounds on the $z$-variables in \Cref{def:polytopes} as well, so that $\mathcal P$ and $\mathcal P_\mathrm{I}$ are indeed polytopes.

Let $[\cdot]_T$ denote the modulo $T$ operator with values in $[0, T)$. Another well-known class of inequalities is the following:

\begin{theorem}[\citealp{nachtigall_cutting_1996}]
\label{thm:change-cycle}
Let $\gamma \in \mathcal C$ and $\alpha_\gamma \coloneqq [-\gamma^\top \ell]_T$. Then the following \emph{change-cycle inequality} holds for all $(x, z) \in \mathcal P_\mathrm{I}$:
\begin{equation}
\label{eq:cc}
(T-\alpha_\gamma) \gamma_+^\top (x - \ell) + \alpha_\gamma \gamma_-^\top (x- \ell) \geq \alpha_\gamma(T-\alpha_\gamma).
\end{equation}
\end{theorem}

A class generalizing both cycle and change-cycle inequalities are the \emph{flip inequalities} introduced by \citet{lindner_determining_2020}.
Let $I$ be a PESP instance and let $F \subseteq A$ be an arbitrary subset of arcs. We construct a new PESP instance $I_F$ from $I$ by ``flipping'' the arcs in $F$: We replace each arc $a = (i, j) \in F$ by an arc $\overline a = (j, i)$, and set $\ell_{\overline a} \coloneqq -u_a$, $u_{\overline a} \coloneqq -\ell_a$, and $w_{\overline a} \coloneqq -w_a$. From any periodic tension $x$ for $I$, we obtain a periodic tension $x_F$ for $I_F$ by defining $x_{F,a} \coloneqq x_a$ for $a \in A \setminus F$, and $x_{F, \overline a} \coloneqq -x_a$ for $a \in F$. In particular, $I$ is feasible if and only if $I_F$ is feasible, and in case of feasibility, both $I$ and $I_F$ have the same optimal objective value. Moreover, for any $\gamma \in \mathcal C$, we obtain an element $\gamma_F$ in the cycle space of $I_F$ by setting $\gamma_{F,a} \coloneqq \gamma_a$ for $a \in A \setminus F$, and $\gamma_{F, \overline a} \coloneqq -\gamma_a$ for $a \in F$. We can hence consider the change-cycle inequality for $\gamma_F$ on $I_F$ and transform it back to $I$:

\begin{theorem}[\citealp{lindner_determining_2020}]
\label{thm:flip-ineq}
Let $\gamma \in \mathcal C$ and $F \subseteq A$. Set
$$\alpha_{\gamma, F} \coloneqq \left[ - \sum_{a \in A \setminus F} \gamma_a \ell_a - \sum_{a \in F} \gamma_a u_a \right]_T.$$
Then the following \emph{flip inequality} holds for all $(x, z) \in \mathcal P_\mathrm{I}$:
\begin{equation}
\label{eq:flip}
\begin{aligned}
&(T - \alpha_{\gamma, F}) \sum_{\substack{a \in A \setminus F:\\\gamma_a > 0}} \gamma_a (x_a - \ell_a) 
+ \alpha_{\gamma, F} \sum_{\substack{a \in A \setminus F:\\\gamma_a < 0}} (-\gamma_a) (x_a - \ell_a)\\
&+ \alpha_{\gamma, F} \sum_{\substack{a \in  F:\\\gamma_a > 0}} \gamma_a (u_a - x_a)
+  (T - \alpha_{\gamma, F}) \sum_{\substack{a \in  F:\\\gamma_a < 0}} (-\gamma_a) (u_a - x_a)
\quad \geq \quad \alpha_{\gamma, F}(T-\alpha_{\gamma, F}).
\end{aligned}
\end{equation}
\end{theorem}

\begin{remark}
\label{rem:flip-special}
The flip inequalities \eqref{eq:flip} for $F = \emptyset$ give exactly the change-cycle inequalities \eqref{eq:cc}. Moreover, by flipping all backward resp.\ all forward arcs of some $\gamma \in \mathcal C$, we obtain Odijk's cycle inequalities \eqref{eq:odijk}. Since the left-hand side of \eqref{eq:flip} is always non-negative for $(x, z) \in \mathcal P$, flip inequalities with $\alpha_{\gamma,F} = 0$ are trivial. Due to symmetry reasons, the flip inequalities for $(\gamma, F)$ and $(-\gamma, F)$ coincide, and $\alpha_{\gamma,F} = T - \alpha_{-\gamma, F}$ when $\alpha_{\gamma, F} \geq 1$.
\end{remark}

\begin{definition}
We define the \emph{flip polytope} as
$$ \mathcal P_\mathrm{flip} \coloneqq \{ (x, z) \in \mathcal P \mid (x, z) \text{ satisfies the flip inequality for all } \gamma \in \mathcal C \text{ and } F \subseteq A \}.$$
\end{definition}
Apart from the trivial relation $\mathcal P_\mathrm{I} \subseteq \mathcal P_{\mathrm{flip}} \subseteq \mathcal P$,
the flip polytope has some interesting properties \citep{lindner_determining_2020}: Every vertex of $\mathcal P_\mathrm{I}$ is a vertex of $\mathcal P_\mathrm{flip}$, but in general not a vertex of $\mathcal P$. Moreover, if $G$ is a cactus graph, i.e., every arc is contained in at most one simple cycle, then $\mathcal P_\mathrm{flip} = \mathcal P_\mathrm{I}$. However, there are PESP instances with $\mu = 2$ and $\mathcal P_\mathrm{flip} \neq \mathcal P_\mathrm{I}$.

\section{The Split Closure of the Periodic Timetabling Polyhedron}
\label{sec:split}
The relation between the periodic timetabling polytope and the flip polytope seems close and deserves more attention. In fact, in this section, we will establish that the flip polytope can be identified with the split closure.

\subsection{Preliminaries}
\label{ss:split-preliminaries}

We will now recall the definition of split inequalities, split disjunctions, and the split closure, following the treatment by \cite{conforti_integer_2014}. To two matrices $A_C \in \mathbb Q^{m \times n}$, $A_I \in \mathbb Q^{m \times p}$ and a vector $b \in \mathbb Q^p$, we associate the mixed-integer set
$$ S \coloneqq  \{(x, z) \in \mathbb R^n \times \mathbb Z^p \mid A_C x + A_I z \leq b\},$$
and the two polyhedra
\begin{align*}
P &\coloneqq \{(x, z) \in \mathbb R^n \times \mathbb R^p \mid A_C x + A_I z \leq b\}, &
P_\mathrm{I} &\coloneqq \conv(S).
\end{align*}

A \emph{split} is a pair $(\beta, \beta_0) \in \mathbb Z^p \times \mathbb Z$. The disjunction
$$\beta^\top z \leq \beta_0 \quad \vee \quad \beta^\top z \geq \beta_0 + 1$$
is satisfied for all $(x, z) \in \conv(S)$ and is called a \emph{split disjunction}. In particular, the polyhedron
$$P^{(\beta, \beta_0)} \coloneqq \conv(\{(x,z) \in P \mid  \beta^\top z \leq \beta_0 \} \cup\{(x,z) \in P \mid  \beta^\top z \geq \beta_0 + 1 \})$$
contains $\conv(S)$. The \emph{split closure} is now defined as
\begin{equation}
    \label{eq:split-def}
    P_\mathrm{split} \coloneqq \bigcap_{(\beta, \beta_0)  \in \mathbb Z^p \times \mathbb Z } P^{(\beta, \beta_0)} = \bigcap_{\beta \in \mathbb Z^p} \conv(\{(x,z) \in P \mid \beta^\top z \in \mathbb Z\}).
\end{equation} 

The split closure $P_\mathrm{split}$ is a polyhedron \citep{cook_chvatal_1990} with the property that $\conv(S) \subseteq P_\mathrm{split} \subseteq P$. It is identical to the closure given by Gomory's mixed-integer (GMI) cuts or mixed-integer rounding (MIR) cuts \citep{nemhauser_recursive_1990}. On the downside, optimization and hence separation are in general NP-hard \citep{caprara_separation_2003}.

The split closure can be described as well in terms of defining inequalities: Define 
\begin{align*}
\Lambda \coloneqq \left\{\lambda \in \mathbb R^m \,\middle|\, \begin{array}{cc}
\lambda^\top A_C = 0, \lambda^\top A_I \in \mathbb Z^p, \lambda^\top b \notin \mathbb Z, \text{and the rows of } (A_C \; A_I)\\  \text{ corresponding to non-zero entries of } \lambda \text{ are linearly independent} \end{array} \right \}.
\end{align*}
Each multiplier vector $\lambda \in \Lambda$ defines the \emph{split inequality}
\begin{equation}
\label{eq:split}
    \frac{\lambda_+^\top (b - A_Cx - A_Iz)}{[\lambda^\top b]_1} + \frac{\lambda_-^\top (b - A_Cx - A_Iz)}{1 - [\lambda^\top b]_1} \geq 1.
\end{equation}
 Here, we decompose $\lambda$ into its positive part $\lambda_+$ and negative part $\lambda_-$, and denote by $[\cdot]_1$ the fractional part in $[0, 1)$.
 The split inequality \eqref{eq:split} for $\lambda \in \Lambda$ is valid for $P^{(\lambda^\top A_I, \lfloor \lambda^\top b \rfloor)}$. Conversely, any facet of $P^{(\beta, \beta_0)}$ is defined by an inequality of the form \eqref{eq:split} for some $\lambda \in \Lambda$ with $\lambda^\top A_I = \beta$ and $\lfloor \lambda^\top b \rfloor = \beta_0$. Consequently,
$$ P_\mathrm{split} = \{(x, z) \in P \mid (x, z) \text{ satisfies the split inequality \eqref{eq:split} for all } \lambda \in \Lambda\}. $$

\subsection{Flipping is Splitting for Periodic Timetabling}

We investigate now the split closure for the cycle-based MIP formulation \eqref{eq:mip-cycle} for the Periodic Event Scheduling Problem. Thus let $(G = (V, A), T, \ell, u, w)$ be a PESP instance, and let $B$ be an integral cycle basis of $G$ with cycle matrix $\Gamma$. Rewriting in the form $A_C x + A_I z \leq b$, the fractional periodic timetabling polytope $\mathcal P$ is defined by

\begin{equation}
    \label{eq:ineq-form}
    \underset{A_C}{\underbrace{%
    \left(
    \begin{array}{c} \Gamma \\ - \Gamma \\ I \\ -I \end{array}
    \right.
    }} \;
    \underset{A_I}{\underbrace{%
    \left.
    \begin{array}{c} -TI \\ TI \\ 0 \\ 0 \end{array}
    \right)
    }}
    \begin{pmatrix} x \\ z \end{pmatrix}
    \leq
    \begin{pmatrix} 0 \\ 0 \\ u \\ -\ell \end{pmatrix},
\end{equation}
where $I$ denotes the identity matrix. We will write a multiplier vector $\lambda \in \Lambda$ as
$$\lambda = (\lambda_1, \lambda_2, \lambda_3, \lambda_4) \in \mathbb R^B \times \mathbb R^B \times \mathbb R^A \times \mathbb R^A$$ corresponding to the four row blocks in \eqref{eq:ineq-form}.

\begin{theorem}
\label{thm:split-is-flip}
Every flip inequality with $\alpha_{\gamma, F} \neq 0$ is a split inequality for the cycle-based MIP formulation of PESP \eqref{eq:mip-cycle} and vice versa. In particular, $\mathcal P_\mathrm{split} = \mathcal P_\mathrm{flip}$.
\end{theorem}
\begin{proof}
 We analyze the set $\Lambda$ for the MIP \eqref{eq:mip-cycle}. For $\lambda \in \Lambda$, we have
 $$ \lambda^\top A_C = (\lambda_1 - \lambda_2)^\top \Gamma + (\lambda_3 - \lambda_4)^ \top \quad \text{and} \quad \lambda^\top A_I = -T(\lambda_1 - \lambda_2)^\top.$$
 
 As $\lambda^T A_I$ is integer, we find that $\gamma \coloneqq T(\lambda_1 - \lambda_2)^\top \Gamma$ is an integer linear combination of the rows of the cycle matrix $\Gamma$, so that $\gamma \in \mathcal C$. From $\lambda^T A_C = 0$ we infer that $\lambda_3 - \lambda_4 = -\gamma/T$. By the linear independence condition, for an arc $a \in A$, not both of $\lambda_{3,a}$ and $\lambda_{4,a}$ can be non-zero. Hence, when we set $F \coloneqq \{a \in A \mid \lambda_{3,a} \neq 0\}$, we have 
 \begin{equation}
 \label{eq:flip-from-split}
 \lambda_{3,a} = \begin{cases}
    -\gamma_a/T & \text{ if } a \in F, \\
    0 & \text{ if } a \in A \setminus F,
 \end{cases} \quad \text{ and } \quad
  \lambda_{4,a} = \begin{cases}
    \gamma_a/T  & \text{ if } a \in A \setminus F, \\
    0 & \text{ if } a \in F.
 \end{cases}
 \end{equation}
 With that, the fractional part $[\lambda^\top b]_1$ evaluates to
 $$ [\lambda^\top b]_1 = [ \lambda_3^\top u - \lambda_4^\top \ell ]_1 = \left[ -\sum_{a \in F} \frac{\gamma_a u_a}{T} - \sum_{a \in A \setminus F} \frac{\gamma_a \ell_a}{T} \right]_1. $$
 Observe that for any $y \in \mathbb R$, we have
 $$ T \left[ \frac{y}{T} \right ]_1 = T \left( \frac{y}{T} - \left\lfloor \frac{y}{T} \right\rfloor \right) = y - T \left\lfloor \frac{y}{T} \right\rfloor = [y]_T,$$
 so that 
 \begin{equation}
 \label{eq:split-denom}
 T [\lambda^\top b]_1 = T [ \lambda_3^\top u - \lambda_4^\top \ell ]_1 = \alpha_{\gamma, F}, \end{equation}
 where $\alpha_{\gamma,F}$ is as in \Cref{thm:flip-ineq},
 and we have $\alpha_{\gamma, F} \neq 0$ because $\lambda^\top b \notin \mathbb Z$.
 
 We now consider the expressions $\lambda_{\pm}^\top(b - A_Cx - A_Iz)$. Since for $(x, z) \in \mathcal P$,
 $$ b - A_Cx -A_Iz = (-\Gamma x + Tz, \Gamma x - Tz, u - x, x - \ell)^\top = (0, 0, u- x, x-\ell)^\top,$$
 we have
 \begin{equation}
    \label{eq:split-num}
 \begin{aligned}
 \lambda_+^\top(b - A_Cx - A_Iz) &= \frac{1}{T} \sum_{\substack{a \in F\\\gamma_a < 0}} \gamma_a(u_a-x_a) + \frac{1}{T} \sum_{\substack{a \in A \setminus F\\\gamma_a > 0}} \gamma_a(x_a-\ell_a), \\
 \lambda_-^\top(b - A_Cx - A_Iz) &= \frac{1}{T} \sum_{\substack{a \in F\\\gamma_a > 0}} \gamma_a(u_a-x_a) + \frac{1}{T} \sum_{\substack{a \in A \setminus F\\\gamma_a < 0}} \gamma_a(x_a-\ell_a).
 \end{aligned}
 \end{equation}
 It is now evident from \eqref{eq:split-denom} and \eqref{eq:split-num} that multiplying the split inequality \eqref{eq:split} for $\lambda$ with $\alpha_{\gamma, F}(T - \alpha_{\gamma, F})$ yields the flip inequality \eqref{eq:flip} for $(\gamma, F)$.
 
 To prove the converse, starting from $\gamma \in \mathcal C$ and $F \subseteq A$ with $\alpha_{\gamma, F} \neq 0$, we define $\lambda_3$ and $\lambda_4$ as in \eqref{eq:flip-from-split}, so that $\lambda_3 - \lambda_4 = -\gamma/T$. Moreover, as \eqref{eq:split-denom} holds, we have that $\lambda^\top b$ is not an integer. Since $\gamma \in \mathcal C$, there is an integral vector $\eta \in \mathbb Z^B$ with $\eta^\top \Gamma = \gamma$. Then $\lambda \coloneqq (\eta/T, 0, \lambda_3, \lambda_4) \in \Lambda$, and the split inequality for $\lambda$ is equivalent to the flip inequality for $(\gamma, F)$.
 \end{proof}

\cite{lindner_determining_2020} proved that $\mathcal P_\mathrm{flip} = \mathcal P_\mathrm{I}$ if the cyclomatic number $\mu$ of $G$ is at most one by analyzing the combinatorial structure of $\mathcal P_\mathrm{flip}$. In terms of the split closure, this result becomes almost trivial:

\begin{corollary}
\label{cor:one-cycle}
Suppose that $\mu \leq 1$. Then $\mathcal P_\mathrm{split} = \mathcal P_\mathrm{I}$.
\end{corollary}
\begin{proof}
This is clear for $\mu = |B| = 0$, as
$$\mathcal P = \mathcal P_\mathrm{split} = \mathcal P_\mathrm{I} = \{(x,z) \in \mathbb R^A \mid \ell \leq x \leq u\} .$$
For $\mu = |B| = 1$, there is only a single integer variable $z$, and by virtue of \eqref{eq:split-def},
\begin{align*}
\mathcal P_\mathrm{split} = \bigcap_{\beta \in \mathbb Z} \conv \left\{(x,z) \in \mathcal P \,\middle|\, \beta z \in \mathbb Z \right \} = \conv \left\{(x,z) \in \mathcal P \,\middle|\, z \in \mathbb Z \right \} = \mathcal P_\mathrm{I}.  & \qedhere
\end{align*}
\end{proof}

\subsection{Chvátal Closure}
\label{ss:chvatal}

For any mixed-integer set $S$ defined by $(A_C, A_I, b)$ with associated polyhedron $P$, one can define the \emph{Chvátal closure} as a ``one-side split closure'' by
$$ P_\mathrm{Ch} \coloneqq \bigcap \left\{P^{(\beta, \beta_0)} \middle|(\beta, \beta_0) \in \mathbb Z^p \times \mathbb Z \text{ s.t. } P \cap \{\beta^\top z \leq \beta_0 \} = \emptyset \text{ or } P \cap \{\beta^\top z \geq \beta_0 + 1 \} = \emptyset\right\},$$
see, e.g., \cite{conforti_polyhedral_2010, conforti_integer_2014}.
It is clear that $P_\mathrm{split} \subseteq P_\mathrm{Ch} \subseteq P$. For Periodic Event Scheduling, we find:

\begin{theorem}
\label{thm:oneside}
    The Chvátal closure of the MIP \eqref{eq:mip-cycle} is given by
    $$ \mathcal P_\mathrm{Ch} = \{ (x,z) \in \mathcal P \mid (x, z) \text{ satisfies the cycle inequality \eqref{eq:odijk} for all } \gamma \in \mathcal C \}.$$
\end{theorem}
\begin{proof}
    We need to determine those $(\beta, \beta_0) \in \mathbb Z^B \times \mathbb Z$ for which  one of $\mathcal P \cap \{\beta^\top z \leq \beta_0 \}$ or $\mathcal P \cap \{\beta^\top z \geq \beta_0 + 1 \}$ is empty. Since $\Gamma x = Tz$ holds for all $(x, z) \in \mathcal P$, we have $\beta^\top z = \frac{\gamma^\top x}{T}$ for $\gamma \coloneqq \beta^\top \Gamma \in \mathcal C$ for an arbitrary choice of $\beta \in \mathbb Z^B$.
    Let
    \begin{align*}
         k_1 &\coloneqq \left  \lceil \min \left \{ \frac{\gamma^\top x}{T}  \,\middle|\, (x, z) \in \mathcal P \right \} \right \rceil
         =  \left  \lceil \frac{\gamma^\top_+ \ell -\gamma^\top_- u}{T} \right \rceil, \\
         k_2 &\coloneqq \left  \lfloor \max \left \{ \frac{\gamma^\top x}{T}  \,\middle|\, (x, z) \in \mathcal P \right \}  \right \rfloor
         = \left\lfloor \frac{\gamma^\top_+ u - \gamma^\top_-\ell}{T} \right\rfloor .
    \end{align*}
    Then $\mathcal P \cap \{ \frac{\gamma^\top x}{T} \leq \beta_0\} = \emptyset$ for all $\beta_0 \leq k_1-1$ and $\mathcal P \cap \{ \frac{\gamma^\top x}{T} \geq \beta_0 +1 \} = \emptyset$ for $\beta_0 \geq k_2$.
    If $k_1 \geq k_2 + 1$, then $\mathcal P_\mathrm{Ch} = \emptyset$, and no $(x, z) \in \mathcal P$ satisfies the cycle inequality \eqref{eq:odijk} for $\gamma$. Otherwise, both polyhedra    $\mathcal P \cap \{ \frac{\gamma^\top x}{T} \geq k_1\}$ and $\mathcal P \cap \{ \frac{\gamma^\top x}{T} \leq k_2\}$ are non-empty, and they are defined by $\mathcal P$ and Odijk's cycle inequalities \eqref{eq:odijk}.

    Moreover, since $\mathcal P \cap \{ \frac{\gamma^\top x}{T} \geq k_1\} \subseteq  \mathcal P \cap \{\frac{\gamma^\top x}{T} \geq \beta_0 \}$ for any $\beta_0 \leq k_1$ and $\mathcal P \cap \{ \frac{\gamma^\top x}{T} \leq k_2\} \subseteq \mathcal P \cap \{ \frac{\gamma^\top x}{T} \leq \beta_0\}$ for $\beta_0 \geq k_2$, we can conclude that $$\bigcap_{\beta_0 \leq k_1-1} P^{(\beta, \beta_0)} = P^{(\beta, k_1-1)} \quad \text{and} \quad \bigcap_{\beta_0 \geq k_2} P^{(\beta,\beta_0)} = P^{(\beta, k_2)}.$$ 

    We conclude that for each $\beta \in \mathbb Z^B$ and $k_1$ and $k_2$ as above,
    \begin{align*}
    & \bigcap  \left\{P^{(\beta, \beta_0)} \middle|\beta_0 \in \mathbb Z \text{ s.t. } P \cap \{\beta^\top z \leq \beta_0 \} = \emptyset \text{ or } P \cap \{\beta^\top z \geq \beta_0 + 1 \} = \emptyset\right\} \\
    &= P^{(\beta, k_1-1)} \cap P^{(\beta, k_2)},
    \end{align*}
    from which the claim follows. 
\end{proof}

\section{Separation of Split Cuts}
\label{sec:sepa}
From a practical point of view, the split closure can be a valuable tool to provide dual bounds for mixed integer programs. Of course, this requires efficient separation methods. As we have established that the split closure is of a specific form in the case of periodic timetabling, we can make use of the combinatorial structure behind flip inequalities to separate cuts. 

\subsection{Simple Cycles}

We show at first that for separating split/flip inequalities, it suffices to consider \emph{simple} oriented cycles, i.e., oriented cycles $\gamma \in \mathcal C \cap \{-1,0,1\}^A$ that yield a simple cycle on the underlying undirected graph of $G$.

\begin{lemma}[Orientation-preserving cycle decomposition]
\label{lem:orient-decomp}
Let $\gamma \in \mathcal C$. Then there are simple oriented cycles $\delta_1, \dots, \delta_r \in \mathcal C$ such that $\gamma = \sum_{k=1}^r \delta_k$, $\gamma_+ = \sum_{k=1}^r \delta_{k,+}$, and $\gamma_- = \sum_{k=1}^r \delta_{k,-}$.
\end{lemma}
\begin{proof}
Let $A^+ \coloneqq \{a \in A \mid \gamma_a > 0\}$ and $A^- \coloneqq \{a \in A \mid \gamma_a < 0\}$ be the set of forward and backward arcs of $\gamma$, respectively. Construct a digraph $G_\gamma$, whose set of arcs $A_\gamma$ is given by
$$ A_\gamma \coloneqq A^+ \cup \{(j, i) \mid (i, j) \in A^- \}.$$
Define $g_{ij} \coloneqq \gamma_{ij}$ if $(i, j) \in A^+$ and $g_{ji} \coloneqq -\gamma_{ij}$ if $(i, j) \in A^-$. Then $g \geq 0$ is a circulation in $G_\gamma$, so that it decomposes into simple directed cycles $d_1, \dots, d_r$. Finally set
$\delta_{k,ij} \coloneqq d_{k,ij}$ if $(i, j) \in A^+$ and $\delta_{k,ij} \coloneqq -d_{k,ji}$ if $(i, j) \in A^-$.
\end{proof}

\begin{theorem}
\label{thm:simple}
Let $F \subseteq A$ and $(x, z) \in \mathcal P$. If $(x, z)$ satisfies all flip inequalities w.r.t.\ $F$ and all simple oriented cycles $\gamma$, then it satisfies all flip inequalities w.r.t.\ $F$ and all $\gamma \in \mathcal C$. In particular,
$$ \mathcal P_\mathrm{split} = \mathcal P_\mathrm{flip} = \left \{ (x, z) \in \mathcal P \,\middle|\, \begin{array}{cc} (x, z) \text{ satisfies the flip inequality} \\ \text{for all simple oriented cycles } \gamma \in \mathcal C \text{ and  all } F \subseteq A \end{array} \right \}. $$
\end{theorem}

\begin{proof}
The inclusion $(\subseteq)$ is clear, it remains to show $(\supseteq)$. Suppose that $(x, z)$ satisfies the flip inequalities w.r.t.\ $F$ and all simple oriented cycles. Moving to the flipped instance $I_F$ as in \Cref{ss:timetabling-polytope}, we can assume that $F = \emptyset$, so that it suffices to consider the change-cycle inequality \eqref{eq:cc} for an arbitrary $\gamma \in \mathcal C$. Let $\gamma = \delta_1 + \dots + \delta_r$ be an orientation-preserving decomposition as in \Cref{lem:orient-decomp}. We proceed by induction on $r$.

If $r \leq 1$, then $\gamma = 0$ or $\gamma$ is simple, and there is nothing to show.

Now assume $r \geq 2$. There is nothing to show if $\alpha_\gamma = 0$, as the left-hand side of the change-cycle inequality is always non-negative. We hence assume $\alpha_\gamma > 0 $.

By induction hypothesis, $(x, z)$ satisfies the change-cycle inequality for the cycles $\delta \coloneqq \delta_1$ and $\varepsilon \coloneqq \delta_2 + \dots + \delta_r$. If $\alpha_\delta = 0$ resp.\ $\alpha_\varepsilon = 0$, then we have $\alpha_\gamma = \alpha_\varepsilon$ resp.\ $\alpha_\gamma = \alpha_\delta$, as $\alpha_\gamma = [\alpha_\delta + \alpha_\varepsilon]_T$. The validity of the change-cycle inequality for $\gamma$ then follows immediately because the right-hand side equals the one for $\varepsilon$ resp.\ $\delta$, while the left-hand side can only become larger.

We are hence left with the case $\alpha_\gamma, \alpha_\delta, \alpha_\varepsilon > 0$.
In this case we can rewrite the change-cycle inequality \eqref{eq:cc} so that it is of the form 
$$ \frac{\gamma_+^\top (x-\ell)}{\alpha_\gamma} + \frac{\gamma_-^\top (x-\ell)}{T - \alpha_\gamma} \geq 1,$$
which will be the key ingredient in our argumentation.
Define 
$$\kappa_\delta \coloneqq \min\left\{ \frac{\alpha_\delta}{\alpha_\gamma},  \frac{T-\alpha_\delta}{T-\alpha_\gamma} \right\} \quad \text{ and } \quad \kappa_\varepsilon \coloneqq \min \left\{ \frac{\alpha_\varepsilon}{\alpha_\gamma},  \frac{T-\alpha_\varepsilon}{T-\alpha_\gamma} \right\}.$$
With $y \coloneqq x- \ell$, using that $(x, z)$ satisfies the change-cycle inequality w.r.t.\ both $\delta$ and $\varepsilon$,
\begin{align*}
    \frac{\gamma_+^\top y}{\alpha_\gamma} + \frac{\gamma_-^\top y}{T - \alpha_\gamma} &= \frac{\delta_+^\top y}{\alpha_\gamma} +  \frac{\delta_-^\top y}{T - \alpha_\gamma}  + \frac{\varepsilon_+^\top y}{\alpha_\gamma} +  \frac{\varepsilon_-^\top y}{T - \alpha_\gamma}  \\
    &=  \frac{\alpha_\delta}{\alpha_\gamma} \frac{\delta_+^\top y}{\alpha_\delta} +  \frac{T-\alpha_\delta}{T-\alpha_\gamma} \frac{\delta_-^\top y}{T - \alpha_\delta}  +  \frac{\alpha_\varepsilon}{\alpha_\gamma} \frac{\varepsilon_+^\top y}{\alpha_\varepsilon} +\frac{T-\alpha_\varepsilon}{T-\alpha_\gamma} \frac{\varepsilon_-^\top y}{T - \alpha_\varepsilon} \\
    &\geq \kappa_\delta \left(  \frac{\delta_+^\top y}{\alpha_\delta} + \frac{\delta_-^\top y}{T - \alpha_\delta}  \right) + \kappa_\varepsilon \left(  \frac{\varepsilon_+^\top y}{\alpha_\varepsilon} + \frac{\varepsilon_-^\top y}{T - \alpha_\varepsilon}  \right)\\
    &  \geq \kappa_\delta + \kappa_\varepsilon,
\end{align*}

\begin{claim}
$\kappa_\delta + \kappa_\varepsilon = 1$.
\end{claim}

If the claim holds, then the change-cycle inequality w.r.t.\ $\gamma$ holds for $(x, z)$, and we are done. Recall that $\alpha_\gamma = [\alpha_\delta + \alpha_\varepsilon]_T$, so that
$$ \alpha_\gamma = \begin{cases}
    \alpha_\delta + \alpha_\varepsilon &\text{ if }  \alpha_\delta + \alpha_\varepsilon < T,\\
    \alpha_\delta + \alpha_\varepsilon - T &\text{ otherwise}.
\end{cases}
$$

If $\alpha_\delta + \alpha_\varepsilon < T$, then $\alpha_\gamma = \alpha_\delta + \alpha_\varepsilon$, hence $\alpha_\delta \leq \alpha_\gamma$ and $T - \alpha_\gamma = T - \alpha_\delta - \alpha_\varepsilon \leq T - \alpha_\delta$, so that $\kappa_\delta =  \alpha_\delta/\alpha_\gamma$. Analogously, $\kappa_\varepsilon =  \alpha_\varepsilon/\alpha_\gamma$, so that $\kappa_\delta + \kappa_\varepsilon = 1$.

In the other case, we have $\alpha_\gamma = \alpha_\delta + \alpha_\varepsilon - T$. From this, we infer $\alpha_\delta = \alpha_\gamma + T - \alpha_\varepsilon \geq \alpha_\gamma$ and 
$T - \alpha_\gamma = 2T - \alpha_\delta - \alpha_\varepsilon \geq (T - \alpha_\delta) + (T - \alpha_\varepsilon) \geq T - \alpha_\delta$, so that $\kappa_\delta = (T-\alpha_\delta)/(T-\alpha_\gamma)$. Analogously, $\kappa_\varepsilon = (T-\alpha_\varepsilon)/(T-\alpha_\gamma)$, so that again $\kappa_\delta + \kappa_\varepsilon = 1$.
\end{proof}

Using \Cref{thm:oneside} and \Cref{rem:flip-special}, \Cref{thm:simple} implies the analogous result for the Chvátal split closure and the cycle inequalities:
\begin{corollary}
Let $(x, z) \in \mathcal P$. If $(x, z)$ satisfies all cycle inequalities w.r.t.\ all simple oriented cycles $\gamma$, then it satisfies all cycle inequalities for all $\gamma \in \mathcal C$. In particular,
  $$ \mathcal P_\mathrm{Ch} = \left \{ (x, z) \in \mathcal P \,\middle|\, \begin{array}{cc} (x, z) \text{ satisfies the cycle inequality} \\ \text{for all simple oriented cycles } \gamma \in \mathcal C \end{array} \right \}.$$
\end{corollary}

\subsection{Separation Hardness}

\cite{lindner_determining_2020} outline a pseudo-polynomial time algorithm based on the dynamic program by \cite{borndorfer_separation_2020} that finds a maximally violated flip inequality (if there is any), i.e., a simple cycle $\gamma$ and a set $F \subseteq A$ such that the difference of the right-hand and left-hand sides of \eqref{eq:flip} is maximum. We prove here that pseudo-polynomial time is best possible unless P $=$ NP:
\begin{theorem}
\label{thm:sepa-hardness}
Given $(x, z) \in \mathcal P$ and $M \geq 0$, it is weakly NP-hard to decide whether there exist a simple cycle $\gamma$ and a subset $F \subseteq A$ such that $(x, z)$ violates the flip inequality for $(\gamma, F)$ by at least $M$.
\end{theorem}
\begin{proof}
We reduce the weakly NP-hard Ternary Partition Problem \citep{borndorfer_separation_2020}: Given $m \in \mathbb N$ and $c \in \mathbb N^m$, is there $a \in \{-1,0,1\}^m$ such that $\sum_{i=1}^m a_i c_i = \pm \frac{1}{2} \sum_{i=1}^m c_i$?
For a Ternary Partition intance $(m, c)$, we define a PESP instance $(G, T, \ell, u, w)$ as follows: The digraph $G = (V, A)$ is given by a complete directed graph on the vertex set
$$ V\coloneqq \{1^+, 1^-, 2^+, 2^-, \dots, m^+, m^-\},$$ where we delete the arcs $(1^-, 1^+), (2^-, 2^+), \dots, (m^-, m^+)$. We set $T \coloneqq \sum_{i=1}^m c_i$ and 
\begin{align*}
    \ell_{i^+ i^-} \coloneqq c_i,  \quad u_{i^+ i^-} \coloneqq T, \quad w_{i^+ i^-} \coloneqq 1 \quad \text{ for all } i \in \{1, \dots, m\}.
\end{align*}
For all other arcs $a$, we set $\ell_a \coloneqq u_a \coloneqq w_a \coloneqq 0$. As for any PESP instance, the optimal solution to the LP relaxation of $\eqref{eq:mip-cycle}$ is given by $x^* = \ell$.

Suppose now that $x^* = \ell$ violates some flip inequality \eqref{eq:flip} for some simple oriented cycle $\gamma$ and some $F \subseteq A$ by at least $M \coloneqq \frac{T^2}{4}$. Since $x^* = \ell$, only arcs in $F$ contribute non-trivially to the left-hand side of $\eqref{eq:flip}$, moreover, these arcs are all of the form $(i^+, i^-)$. We hence obtain
\begin{equation}
\label{eq:max-viol-cc}
\alpha_{\gamma,F} \sum_{(i^+, i^-) \in F, \gamma_{i^+ i^-} = 1} (T-c_i) + (T- \alpha_{\gamma,F}) \sum_{(i^+, i^-) \in F, \gamma_{i^+ i^-} = -1} (T-c_i) \leq \alpha_{\gamma,F}(T-\alpha_{\gamma,F}) - M,
\end{equation}
where $\alpha_{\gamma,F} = [-\sum_{(i^+, i^-) \notin F} \gamma_{i^+ i^-} c_i]_T$. As the left-hand side of \eqref{eq:flip} is non-negative, we have that $\alpha_{\gamma,F}(T-\alpha_{\gamma,F}) \geq M = T^2/4$, which implies $\alpha_{\gamma,F} = T/2$. Set $a_i \coloneqq -\gamma_{i^+, i^-}$ for all $i \in \{1, \dots, m\}$. Then $[\sum_{i=1}^m a_i c_i]_T = \alpha_{\gamma, F} = T/2$, and as $-T \leq \sum_{i=1}^m a_i c_i \leq T$, we find that $\sum_{i=1}^m a_i c_i = \pm T/2$. In particular, a violated flip inequality leads to a positive answer to the Ternary Partition instance.

Conversely, suppose that there is $a \in \{-1, 0, 1\}^m$ such that $\sum_{i=1}^m a_i c_i = \pm T/2$. Construct a simple oriented cycle $\gamma$ with $\gamma_{i^+ i^-} \coloneqq -a_i$ for all $i \in \{1, \dots, m\}$. Then $\alpha_{\gamma, \emptyset} = T/2$, and the flip inequality for $\gamma$ and $F = \emptyset$ (i.e., the change-cycle inequality for $\gamma$) is violated by at least $T^2/4 = M$, because the left-hand side of $\eqref{eq:max-viol-cc}$ vanishes. 
\end{proof}

In practice, the dynamic program indicated in \citep{lindner_determining_2020} consumes too much memory. It is therefore advantageous to switch to a cut-generating MIP. \cite{balas_optimizing_2008} describe a parametric MIP with a single parameter $\theta \in [0, 1]$ for this purpose. We are however in a better situation: Translating to periodic timetabling via \Cref{thm:split-is-flip}, the parameter $\theta$ essentially corresponds to $\alpha_{\gamma, F}$, which is always an integer between $0$ and $T-1$. This means that the parametric MIP can be replaced by a finite sequence of standard IPs for each such integer $\alpha_{\gamma, F}$. The formulation of the IP \eqref{eq:pmip} is straightforward from the definition \eqref{eq:flip} of flip inequalities:

\begin{theorem}
\label{thm:sepa-parametric-mip}
Let $(x, z) \in \mathcal P \setminus \mathcal P_\mathrm{flip}$ and $\alpha \in \{1, \dots, T$\}. Then a maximally violated flip inequality w.r.t.\ $(x, z)$ with $\alpha_{\gamma, F} = \alpha$ among all oriented cycles $\gamma$ and all $F \subseteq A$ is found by the following integer program:
\begin{equation}
\label{eq:pmip}
\begin{aligned}
    & \text{Minimize} & (T- \alpha)\sum_{a \in A} (x_a - \ell_a) y^+_a +  \alpha \sum_{a \in A}  (x_a - \ell_a) y^-_a \\
    & & + \alpha \sum_{a \in A} (x_a - \ell_a) f^+_a + (T- \alpha) \sum_{a \in A} (u_a - x_a) f^-_a  \\
    & \text{s.t.} & \sum_{a \in A} \ell_a (y^-_a - y^+_a) + \sum_{a \in A} u_a (f^-_a - f^+_a) + k T &= \alpha \\
    & & \sum_{a \in \delta^+(v)} \gamma_a - \sum_{a \in \delta^-(v)} \gamma_a &= 0, & \quad v \in V,\\
    & & f^+ - f^- + y^+ - y^- &= \gamma,\\
    & & 0 \leq f^+ + f^- + y^+ + y^- &\leq 1,\\
    & & f^+, f^-, y^+, y^- &\in \{0, 1\}^A,\\
    & & \gamma &\in \{-1, 0, 1\}^A,\\
    & & k &\in \mathbb Z.
\end{aligned}
\end{equation}
\end{theorem}


Any feasible solution of \eqref{eq:pmip} with objective value less than $\alpha(T-\alpha)$ will produce a violated flip inequality. Recall from \Cref{rem:flip-special} that flip inequalities with $\alpha_{\gamma, F} = 0$ are trivial and cannot be violated, and that due to symmetry, it is not necessary to consider the IP \eqref{eq:pmip} for $\alpha \geq T/2$.

\subsection{Separation for a Fixed Cycle}
\label{sepa:fixed}

We discuss now how to find a maximally violated flip inequality in linear time when the cycle $\gamma$ is already fixed. To this end, we take the perspective of split cuts. Consider again a mixed-integer set defined by $(A_C, A_I, b)$ and the associated polyhedron $P = \{A_C x + A_I z \leq b\}$. When a split $(\beta, \beta_0)$ is fixed, then the separation problem on $P^{(\beta, \beta_0)}$ can be solved as follows \citep{conforti_integer_2014,bonami_optimizing_2012,balas_lift-and-project_1993}: Given $(x, z) \in P$, check whether $\beta^\top z \leq \beta_0$ or $\beta^\top z \geq \beta_0 + 1$. If yes, then $(x,  z) \in P^{(\beta, \beta_0)}$. Otherwise, solve the linear program
\begin{equation}
    \label{eq:cut-gen}
    \begin{aligned}
        & \text{Minimize} & (s - t)^\top b + \frac{1}{\beta^\top z - \beta_0} \cdot t^\top (b - A_C x - A_I z) \\
        & \text{s.t.} & (s - t)^\top A_C &= 0 \\
        & & (s - t)^\top A_I &= \beta^\top \\
        & & s, t &\geq 0.
    \end{aligned}
\end{equation}
If the value of \eqref{eq:cut-gen} is at least $\beta_0 + 1$, then $(x, z) \in P^{(\beta, \beta_0)}$, otherwise it is not. In the latter case, if we take a basic optimal solution $(s^*, t^*)$, then $(x, z)$ is separated by the split inequality w.r.t.\ $s^* - t^*$. This cut-generating LP \eqref{eq:cut-gen} finds a maximally violated split inequality in the following sense:

\begin{lemma}
\label{lem:max-violation}
Suppose that $(x, z) \in P \setminus P^{(\beta, \beta_0)}$. Let $(s^*, t^*)$ be an optimal basic solution of \eqref{eq:cut-gen}, $\lambda^* \coloneqq s^*-t^*$. Then
\begin{equation}
\label{eq:max-violation}
[\lambda^{*\top} b]_1 (1 - [\lambda^{*\top} b]_1)  -  (1 - [\lambda^{*\top} b]_1) \lambda_+^{*\top} (b - A_Cx - A_Iz) - [\lambda^{*\top} b]_1 \lambda_-^{*\top} (b - A_Cx - A_Iz)
\end{equation}
is maximum among all $\lambda = s - t$ such that $(s, t)$ is feasible for \eqref{eq:cut-gen} and $\lambda^\top b \in [\beta_0, \beta_0 + 1)$.
\end{lemma}
\begin{proof}
Since $(s^*, t^*)$ is basic, we have $\lambda^*_+ = s$ and $\lambda^*_- = t$. As $(x, z) \in P \setminus P^{(\beta, \beta_0)}$, $\beta^\top z - \beta_0 > 0$. Then $\lambda^*$ maximizes
$$  - (\beta^\top z - \beta_0) \lambda^\top b - \lambda_-^\top (b - A_C x - A_I z).$$
Adding a constant term, $\lambda^*$ also maximizes
\begin{align*}
&(\beta^\top z - \beta_0) (\beta_0 + 1) - (\beta^\top z - \beta_0) \lambda^\top b - \lambda_-^\top (b - A_C x - A_I z) \\
&= (\beta^\top z - \beta_0) (\beta_0 + 1 -\lambda^\top b) - \lambda_-^\top (b - A_C x - A_I z) .
\end{align*}
Observing that $\lambda^\top (b - A_C x - A_I z) = \lambda^\top b - \beta^\top z$, this is the same as
\begin{align*}
&
= (\lambda^\top b - \lambda^\top(b - A_C x - A_I z) - \beta_0) (\beta_0 + 1 -\lambda^\top b) - \lambda_-^\top (b - A_C x - A_I z) \\
&= (\lambda^\top b - \beta_0) (\beta_0 + 1 - \lambda^\top b) - (\beta_0 + 1 - \lambda^\top b) \lambda^\top(b - A_C x - A_I z) - \lambda_-^\top(b - A_C x - A_I z) \\
&=  (\lambda^\top b - \beta_0) (\beta_0 + 1 - \lambda^\top b) - (\beta_0 + 1 - \lambda^\top b) \lambda_+^\top(b - A_C x - A_I z) \\
& \quad - (\lambda^\top b - \beta_0 ) \lambda_-^\top(b - A_C x - A_I z).
\end{align*}
Since $\lambda^\top b \in [\beta_0, \beta_0 + 1)$, $[\lambda^\top b]_1 = \lambda^\top b - \beta_0$ and $1-[\lambda^\top b]_1 = \beta_0 + 1 - \lambda^\top b$, and we arrive at \eqref{eq:max-violation}.
\end{proof}

Note that the condition $\lambda^\top b \in [\beta_0, \beta_0 +1)$ in \Cref{lem:max-violation} is no restriction, since it suffices to consider $\lambda$ for which  $\lfloor \lambda^\top b \rfloor = \beta_0$ (cf.\ \Cref{sec:split}). We obtain the following in the context of periodic timetabling:

\begin{theorem}
\label{thm:sepa-fixed-cycle}
Let $\mathcal P$ be a fractional periodic timetabling polytope. Let $(x, z) \in \mathcal P$, $\gamma \in \mathcal C$ with $\gamma^\top x \notin T \mathbb Z$, and set $g \coloneqq T/[- \gamma^\top x]_T$. Then the flip inequality w.r.t.\ $\gamma$ and
\begin{align*}
    F \coloneqq \{a \in A \mid \gamma_a > 0 \text{ and } u_a - \ell_a \geq g(u_a - x_a) \} \cup 
    \{a \in A \mid \gamma_a < 0 \text{ and } u_a - \ell_a \leq g(x_a - \ell_a)  \}
\end{align*}
is maximally violated by $(x, z)$ among the flip inequalities w.r.t.\ $\gamma$. In particular, a maximally violated flip inequality w.r.t.\ $\gamma$ can be found in $O(|\gamma|)$ time.
\end{theorem}

\begin{proof}
We first write down the cut-generating LP \eqref{eq:cut-gen} for the PESP situation \eqref{eq:ineq-form}:

\begin{align*}
    & \text{Minimize} & s_3^\top u - t_3^\top u - s_4^\top \ell + t_4^\top \ell + \frac{1}{\beta^\top z - \beta_0} \cdot (t_3^\top (u - x) + t_4^\top(x - \ell)) \\
    & \text{s.t.} & (s_1 - t_1 - s_2 + t_2)^\top \Gamma + s_3^\top - t_3^\top - s_4^\top + t_4^\top &= 0, \\
    & &  - s_1 + t_1 + s_2 - t_2 &= \frac{\beta}{T},\\
    & & s_1, s_2, s_3, s_4, t_1, t_2, t_3, t_4 &\geq 0.
\end{align*}
Recall from \Cref{thm:split-is-flip} that a flip inequality w.r.t.\ $\gamma$ corresponds to a split inequality derived from $P^{(\beta, \beta_0)}$ with $\beta^\top \Gamma = -\gamma$. Since $(x, z) \in \mathcal P$, we have $\beta_0 = \lfloor \beta^\top z \rfloor = \lfloor -\gamma^\top x/T \rfloor$.
Eliminating the variables $s_1, s_2, t_1, t_2$, and setting
$$ g \coloneqq \frac{1}{\beta^\top z - \beta_0} = \frac{1}{[ \beta^\top z]_1 } =  \frac{1}{[ \beta^\top \Gamma x/T ]_1}= \frac{1}{[-\gamma^\top x/T]_1} = \frac{T}{[-\gamma^\top x]_T},$$
this becomes
\begin{align*}
    & \text{Minimize} & s_3^\top u - t_3^\top u - s_4^\top \ell + t_4^\top \ell + g \cdot (t_3^\top (u - x) + t_4^\top(x - \ell)) \\
    & \text{s.t.} & s_3 - t_3 - s_4 + t_4 &= -\frac{\gamma}{T}, \\
    & & s_3, s_4, t_3, t_4 &\geq 0.
\end{align*}
This linear program is trivial to solve: In each basic solution, for each arc $a \in A$ at most one of $s_{3,a}, s_{4,a}, t_{3,a}, t_{4,a}$ will be non-zero, and $s_{3,a} = s_{4,a} = t_{3,a} = t_{4,a} = 0$ for all $a \in A$ with $\gamma_a = 0$. We examine the contribution to the objective for each arc $a$ in $\gamma$ in such a basic solution:

If $\gamma_a > 0$, then either $t_{3,a}> 0$ or $s_{4,a} > 0$. In the first case, the contribution to the objective is $ \gamma_a(g(u_a - x_a) - u_a)/T $, otherwise $-\gamma_a\ell_a/T$. Otherwise, if $\gamma_a < 0$, then either $s_{3,a} > 0$ or $t_{4,a} > 0$, the contribution being $-\gamma_a u_a/T$ resp.\ $ -\gamma_a (\ell_a + g(x_a - \ell_a))/T $. In particular, an optimal solution is given by
\begin{align*}
    t_{3,a} &\coloneqq \frac{\gamma_a}{T} & \text{ for all } a \text{ s.t. } \gamma_a > 0 \text{ and } -\ell_a \geq g(u_a - x_a) - u_a,\\
    s_{4,a} &\coloneqq \frac{\gamma_a}{T} & \text{ for all } a \text{ s.t. } \gamma_a > 0 \text{ and } -\ell_a < g(u_a - x_a) - u_a,\\
    s_{3,a} &\coloneqq -\frac{\gamma_a}{T} & \text{ for all } a \text{ s.t. } \gamma_a < 0 \text{ and } u_a \leq g(x_a - \ell_a) + \ell_a,\\
    t_{4,a} &\coloneqq -\frac{\gamma_a}{T} & \text{ for all } a \text{ s.t. } \gamma_a < 0 \text{ and } u_a > g(x_a - \ell_a) + \ell_a,
\end{align*}
and $s_{3,a} \coloneqq s_{4,a} \coloneqq t_{3,a} \coloneqq t_{4,a} \coloneqq 0$ otherwise. The cut derived from this solution is the split inequality for $\lambda = s - t$, which by \Cref{thm:split-is-flip} corresponds to the flip inequality for $\gamma$ and 
\begin{align*}
    F &= \{a \in A \mid \lambda_{3,a} \neq 0\} \\
    &= \{a \in A \mid \gamma_a > 0 \text{ and } u_a - \ell_a \geq g(u_a - x_a) \} \cup 
    \{a \in A \mid \gamma_a < 0 \text{ and } u_a - \ell_a \leq g(x_a - \ell_a)  \}.
\end{align*}
Observe that by \eqref{eq:split-denom}, $T[\lambda^\top b]_1 = \alpha_{\gamma, F}$. Using \eqref{eq:split-num} and multiplying \eqref{eq:max-violation} with $T^2$ therefore yields the violation of the flip inequality w.r.t.\ $(\gamma, F)$. By \Cref{lem:max-violation}, we conclude that the violation is indeed maximal.
\end{proof}

\section{Comparing Split Closures}
\label{sec:comparing}

Recall that the Periodic Event Scheduling Problem can be formulated in two ways as a MIP, where the incidence-based formulation \eqref{eq:mip-incidence} is essentially a special case of the cycle-based formulation \eqref{eq:mip-cycle} by virtue of \Cref{rem:incidence-is-cycle}. The methods of \Cref{sec:split} therefore apply to both formulations, and the question arises whether one of the two split closures is stronger. We will show that both closures are in fact of the same strength in \Cref{ss:projection}.

Typically, the integer variables in both formulations are general. However, under certain circumstances, the periodic offset variables $p_a$ in \eqref{eq:mip-incidence} can be assumed to be binary \citep{liebchen_periodic_2006}. We will discuss \Cref{ss:binarization} how to achieve binary variables by a subdivision procedure. We will show that this binarization approach does not lead to a stronger split closure.

We show in \Cref{ss:product} that split closures commute with Cartesian products, which means in the PESP situation that the split closures can be considered on blocks of $G$ individually.

However, to be able to compare split closures of different polyhedra, we need to develop a few technicalities first in \Cref{ss:mip-compatible-maps}. 

\subsection{Mixed-Integer-Compatible Maps}
\label{ss:mip-compatible-maps}

We begin with two mixed-integer sets
$$ S_i := \{(x,z) \in \mathbb R^{n_i} \times \mathbb Z^{p_i} \mid A_C^i x + A_I^i z \leq b^i\}, \quad i \in \{1,2\},$$
and the associated polyhedra
\begin{align*}
    P_i &:= \{(x,z) \in \mathbb R^{n_i} \times \mathbb R^{p_i} \mid A_C^i x + A_I^i z \leq b^i\},  &
    (P_{i})_\mathrm{I} &:= \conv(S_i), &\quad i \in \{1,2\}.
\end{align*}

\begin{definition}
A map $\varphi: \mathbb R^{n_1} \times \mathbb R^{p_1} \to  \mathbb R^{n_2} \times \mathbb R^{p_2}$ is \emph{mixed-integer-compatible} if $\varphi$ is affine and $\varphi(\mathbb R^{n_1} \times \mathbb Z^{p_1}) \subseteq \mathbb R^{n_2} \times \mathbb Z^{p_2}$.
\end{definition}

In particular, if $\varphi(P_1) \subseteq P_2$ for a mixed-integer-compatible map $\varphi$, then $\varphi(S_1) \subseteq S_2$ and $\varphi((P_1)_\mathrm{I}) \subseteq (P_2)_\mathrm{I}$.

\begin{lemma}
\label{lem:mixed-integer-compatible-dual-map}
    Let $\psi: \mathbb R^{n_1} \times \mathbb R^{p_1} \to  \mathbb R^{n_2} \times \mathbb R^{p_2}$ be a linear map and let $\psi^*: \mathbb R^{n_2} \times \mathbb R^{p_2} \to  \mathbb R^{n_1} \times \mathbb R^{p_1}$ be the corresponding dual linear map, identifying dual vector spaces choosing standard bases. Then the following are equivalent:
    \begin{enumerate}[\normalfont (1)]
        \item $\psi$ is mixed-integer-compatible.
        \item $\psi(\mathbb R^{n_1} \times \{0\}) \subseteq \mathbb R^{n_2} \times \{0\}$ and $\psi^*(\{0\} \times \mathbb Z^{p_2}) \subseteq \{0\} \times \mathbb Z^{p_1}$.
    \end{enumerate}
\end{lemma}
\begin{proof}
    $(1) \Rightarrow (2)$: For the first statement consider for $i \in [n_1]$ the $i$-th standard basis vector $e_i \in \mathbb R^{n_1}$. Then $\psi(e_i, 0) = (x, z)$ for some  $x \in \mathbb R^{n_2}$ and $z \in \mathbb R^{p_2}$. But as $\psi$ is linear and mixed-integer-compatible, $\psi(\lambda e_i, 0) = (\lambda x, \lambda z)$ with $\lambda z \in \mathbb Z^{p_2}$ for all $\lambda \in \mathbb R$, so that $z = 0$.

    For the second statement, consider for $j \in [p_2]$ the $j$-th standard basis vector $e_j$. Then for $i \in [n_1]$, the $i$-th coordinate of $\psi^*(0, e_j)$ is given by $(0, e_j)^\top \psi(e_i, 0) = 0$ by the first statement. For $i \in [p_1]$, the $(n_1 + i)$-th coordinate of $\psi^*(0, e_j)$ is given by $(0, e_j)^\top \psi(0, e_i)$, which is integral as $\psi$ is mixed-integer-compatible.
    
    $(2) \Rightarrow (1)$: Let $(x, z) \in \mathbb R^{n_1} \times \mathbb Z^{p_1}$. Then $\psi(x, z) = \psi(x, 0) + \psi(0, z)$, so using linearity and the first statement in (2), it suffices to consider $\psi(0, e_i)$ for $i \in [p_1]$. But now for $j \in [p_2]$, the $(n_1 + i)$-th coordinate of $\psi^*(0, e_j)$ is integral by the second statement in (2), and since it is given by $(0, e_j)^\top \psi(0, e_i)$, we conclude  that the $(n_2 + j)$-th coordinate of $\psi(0, e_i)$ is integer. Consequently, $\psi$ must be mixed-integer-compatible.    
\end{proof}

The following is a generalization of Theorem~1 in \citep{dash_binary_2018}.
\begin{theorem}
\label{thm:split-descendance}
    Let $\varphi$ be a mixed-integer-compatible map with $\varphi(P_1) \subseteq P_2$. Then $\varphi((P_1)_\mathrm{split}) \subseteq (P_2)_\mathrm{split}$.
\end{theorem}
\begin{proof}
    Consider $(x_1, z_1) \in (P_1)_\mathrm{split}$ and $\beta_2 \in \mathbb Z^{p_2}$. We need to show that $\varphi(x_1, z_1)$ is a convex combination of points $(x_2^i, z_2^i) \in P_2$ with $\beta_2^\top z_2^i$ integral.   
    Since $\varphi$ is mixed-integer-compatible, the last $p_2$ entries of $\varphi(0, 0)$ are integral, and so the linear map $\psi := \varphi - \varphi(0,0)$ is mixed-integer-compatible as well. By~\Cref{lem:mixed-integer-compatible-dual-map}, $\psi^*(0, \beta_2) = (0, \beta_1)$ for some $\beta_1 \in \mathbb Z^{p_1}$. Since $(x_1, z_1) \in (P_1)_\mathrm{split}$, it is a convex combination of $(x_1^i, z_1^i) \in P_1$ with $\beta_1^\top z_1^i \in \mathbb Z$. Write
    $$(x_2^i, z_2^i) := \varphi(x_1^i, z_1^i) = \psi(x_1^i, z_1^i) + \varphi(0, 0) \in P_2.$$
    Then
    \begin{align*}
        \beta_2^\top z_2^i = (0, \beta_2)^\top (x_2^i, z_2^i) &= (0, \beta_2)^\top \psi(x_1^i, z_1^i) + (0, \beta_2)^\top \varphi(0, 0) \\
        &= \psi^*(0, \beta_2)^\top (x_1^i, z_1^i) +  (0, \beta_2)^\top \varphi(0, 0) \\
        &=  (0, \beta_1)^\top (x_1^i, z_1^i) +  (0, \beta_2)^\top \varphi(0, 0)  \\
        &= \beta_1^\top z_1^i + (0, \beta_2)^\top \varphi(0, 0) \\
        &\in \mathbb Z.
    \end{align*}
     As $\varphi$ is affine and hence preserve convex combinations, $\varphi(x_1, z_1)$ is a convex combination of the $(x_2^i, z_2^i) \in P_2$.
\end{proof}

\begin{example}
An example for a mixed-integer-compatible map is provided by the change of the cycle basis in the context of periodic timetabling. Let $I = (G, T, \ell, u, w)$ be a PESP instance and let $\Gamma, \Gamma'$ be two cycle matrices of integral cycle bases of $G$. As in \Cref{rem:basis-indep}, there is an unimodular matrix $U$ such that $\Gamma' = U\Gamma$. The map $\varphi: (x, z) \mapsto (x, Uz)$ maps the fractional periodic timetabling polytope $\mathcal P_1$ defined by $\Gamma$ to the fractional periodic timetabling polytope $\mathcal P_2$ defined by $\Gamma'$. The map $\varphi$ is clearly linear and maps mixed-integer points to mixed-integer points, so that $\varphi$ is mixed-integer-compatible by definition. We conclude that $\varphi((\mathcal P_1)_\mathrm{split}) \subseteq \mathcal (P_2)_\mathrm{split}$. Since $U$ is unimodular, $\varphi$ has a mixed-integer compatible inverse, so that $\varphi$ provides a ``mixed-integer'' isomorphism of $(P_1)_\mathrm{split}$ with $(P_2)_\mathrm{split}$.
\end{example}

\subsection{Split Closure of Cartesian Products}
\label{ss:product}

As first application of mixed-integer-compatible maps, we prove that split closures are compatible with Cartesian products.

\begin{theorem}
\label{thm:split-product}
    Consider two mixed-integer sets
    $$ S_i = \{(x,z) \in \mathbb R^{n_i} \times \mathbb Z^{p_i} \mid A_C^i x + A_I^i z \leq b^i\}, \quad i \in \{1,2\},$$
    and the associated polyhedra
    $$ P_i := \{(x,z) \in \mathbb R^{n_i} \times \mathbb R^{p_i} \mid A_C^i x + A_I^i z \leq b^i\}, \quad i \in \{1, 2\}.$$
    Then $(P_1 \times P_2)_\mathrm{split} = (P_1)_\mathrm{split} \times (P_2)_\mathrm{split}$.
\end{theorem}
\begin{proof}
    We first prove $(P_1)_\mathrm{split} \times (P_2)_\mathrm{split}  \subseteq (P_1 \times P_2)_\mathrm{split}$ using the characterization \eqref{eq:split-def}:
    \begin{align*}
        &(P_1)_\mathrm{split} \times (P_2)_\mathrm{split} \\
        &= \left( \bigcap_{\beta_1 \in \mathbb Z^{p_1}} \conv(\{(x_1,z_1) \in P_1 \mid \beta_1^\top z_1 \in \mathbb Z\}) \right) \times \left( \bigcap_{\beta_2 \in \mathbb Z^{p_2}} \conv(\{(x_2,z_2) \in P_2 \mid \beta_2^\top z_2 \in \mathbb Z\}) \right) \\
        &= \bigcap_{\beta_1 \in \mathbb Z^{p_1}}  \bigcap_{\beta_2 \in \mathbb Z^{p_2}} \left( \conv(\{(x_1,z_1) \in P_1 \mid \beta_1^\top z_1 \in \mathbb Z\}) \times \conv(\{(x_2,z_2) \in P_2 \mid \beta_2^\top z_2 \in \mathbb Z\}) \right) \\
        &= \bigcap_{(\beta_1, \beta_2) \in \mathbb Z^{p_1} \times \mathbb Z^{p_2}}   \conv(\{(x_1,z_1) \in P_1 \mid \beta_1^\top z_1 \in \mathbb Z\} \times \{(x_2,z_2) \in P_2 \mid \beta_2^\top z_2 \in \mathbb Z\}) \\
        &= \bigcap_{(\beta_1, \beta_2) \in \mathbb Z^{p_1} \times \mathbb Z^{p_2}}   \conv(\{(x_1,z_1,x_2,z_2) \in P_1 \times P_2 \mid \beta_1^\top z_1 \in \mathbb Z, \beta_2^\top z_2 \in \mathbb Z\}) \\
        &\subseteq \bigcap_{(\beta_1, \beta_2) \in \mathbb Z^{p_1} \times \mathbb Z^{p_2}}   \conv(\{(x_1,z_1,x_2,z_2) \in P_1 \times P_2 \mid (\beta_1, \beta_2)^\top (z_1, z_2) \in \mathbb Z\}) \\
        &= (P_1 \times P_2)_\mathrm{split}.
    \end{align*}
    To show the reverse inclusion, we consider the natural projections $\varphi_i: P_1 \times P_2 \to P_i$ for $i \in \{1, 2\}$. Both $\varphi_i$ are mixed-integer-compatible, so that by \Cref{thm:split-descendance}, $\varphi_i((P_1 \times P_2)_\mathrm{split}) \subseteq (P_i)_{\mathrm{split}}$. In particular, the map $(\varphi_1, \varphi_2)$, which is the identity map, maps $(P_1 \times P_2)_\mathrm{split}$ into $(P_1)_\mathrm{split} \times (P_2)_\mathrm{split}$. 
\end{proof}

We apply now \Cref{thm:split-product} to periodic timetabling. Consider for an arbitrary digraph $G$ its decomposition into blocks. Since each cycle is part of a unique block, the cycle space of $G$ decomposes into the direct sum of the cycle spaces of its blocks. This has the consequence that any cycle matrix $\Gamma$ of $G$ has a block structure as well, so that the fractional periodic timetabling polytope $\mathcal P$ is the Cartesian product of the fractional periodic timetabling polytopes associated to the subinstances of each block.

\begin{theorem}[cf.\ \citealp{lindner_determining_2020}]
\label{thm:split-block}
    If $G_1, \dots, G_k$ are the blocks of $G$ and $\mathcal P_1, \dots, \mathcal P_k$ are the fractional periodic timetabling polytopes of the subinstances of $G_1, \dots, G_k$, respectively, then 
    $$ \mathcal P_\mathrm{split} = (\mathcal P_1)_\mathrm{split} \times \dots \times (\mathcal P_k)_\mathrm{split}.$$
    In particular, if $G$ is a cactus graph, then $\mathcal P_\mathrm{split} = \mathcal P_\mathrm{I}$.
\end{theorem}
\begin{proof}
    By the above discussion, this is a direct consequence of \Cref{thm:split-product}.
    If $G$ is a cactus graph, then each block satisfies $\mu \leq 1$. It remains to apply \Cref{cor:one-cycle}.
\end{proof}

\subsection{Incidence-Based vs.\ Cycle-Based Formulation}
\label{ss:projection}

Recall from \Cref{rem:incidence-is-cycle} that the incidence-based formulation \eqref{eq:mip-incidence} of a PESP instance is identical to a particular cycle-based formulation \eqref{eq:mip-cycle} of an augmented instance, where the augmentation consists in successively adding arcs $a$ with $\ell_a = 0$ and $u_a = T-1$. Such arcs with $u_a - \ell_a = T-1$ are sometimes called \emph{free} (e.g., \citealt{goerigk_improved_2017}), as they do not impact the feasibility of a PESP instance. The augmentation procedure in \Cref{rem:incidence-is-cycle} hence decomposes as a sequence of $|V|$ \emph{simple free augmentations}, which we formally define as follows:

\begin{definition}
Let $I = (G, T, \ell, u, w)$ be a PESP instance. Let $I' = (G', T, \ell', u', w')$ be a PESP instance such that $I$ arises from $I'$ by deleting a \emph{free} arc $\overline{a}$, i.e., $u'_{\overline{a}} - \ell'_ {\overline{a}} = T-1$. We say that $I'$ is a \emph{simple free augmentation} of $I$ by $\overline{a}$.
\end{definition}

We will first investigate a trivial case of a simple augmentation $I'$ of $I$ by $\overline{a}$: If $\overline{a}$ is a bridge, then $\overline{a}$ constitutes a block of $G'$, so that we conclude by \Cref{thm:split-block} that
\begin{equation}
\label{eq:split-augmentation-bridge}
\mathcal P'_\mathrm{split} = \mathcal P_\mathrm{split} \times [\ell'_{\overline{a}}, u'_{\overline{a}}]_\mathrm{split} = \mathcal P_\mathrm{split} \times [\ell'_{\overline{a}}, u'_{\overline{a}}],    
\end{equation}
where $\mathcal P$ and $\mathcal P'$ are the fractional periodic tension polytopes of $I$ and $I'$, respectively. Since $\overline{a}$ is a bridge, any cycle basis for $G$ is a cycle basis for $G'$, so that the choice of any integral cycle basis yields a natural  projection $\mathcal P'_\mathrm{split} \to \mathcal P_\mathrm{split}, (x, x_{\overline{a}}, z) \mapsto (x, z)$,
which is well-defined and surjective by \eqref{eq:split-augmentation-bridge}. Thus any split inequality for $I'$ is trivially a split inequality for $I$ and vice versa.

We will hence turn our interest to the more interesting case that $\overline{a}$ is not a bridge. We start with an observation about cycle bases:

\begin{lemma}
\label{lem:augmented-cycle-basis}
    Let $I'$ be a simple free augmentation of $I$ by $\overline{a}$ such that $\overline{a}$ is not a bridge of $G'$.
    Then there is an integral cycle basis $B$ of $G$ and an oriented cycle $\overline{\gamma}$ such that $B' := B \cup \{\overline{\gamma}\}$ is an integral cycle basis of $G'$ and $\overline{a} \in \overline{\gamma}$.
\end{lemma}
\begin{proof}
    Since $\overline{a}$ is not a bridge, $G$ and $G'$ have the same set of nodes, so that any spanning tree of $G$ is a spanning tree of $G'$. Hence, if $B$ is any fundamental cycle basis of $G$, we can augment $B$ by the fundamental cycle $\overline{\gamma}$ induced by $\overline{a}$ in $G'$.
\end{proof}

Choose cycle bases $B$, $B'$ and an oriented cycle $\overline{\gamma}$ as in \Cref{lem:augmented-cycle-basis}. We assume that $\mathcal P$ is defined using the cycle matrix $\Gamma$ of $B$, and that $\mathcal P'$ is defined using the cycle matrix $\Gamma'$ of $B'$, so that $\Gamma'$ arises from $\Gamma$ by appending the row $\overline{\gamma}^\top$.

\begin{lemma}
\label{lem:projection-compatible}
Let $I'$ be a simple free augmentation of $I$ by $\overline{a}$ such that $\overline{a}$ is not a bridge of $G'$.
    The natural projection
    $ \varphi: \mathcal P' \to \mathcal P, (x, x_{\overline a}, z, z_{\overline{\gamma}}) \mapsto (x, z)$
    is mixed-integer-compatible.
    In particular, $\varphi(\mathcal P'_\mathrm{split}) \subseteq \mathcal P_\mathrm{split}$.
\end{lemma}
\begin{proof}
    The map $\varphi$ is linear and maps mixed-integer points to mixed-integer points. That $\varphi$ descends to split closures follows from \Cref{thm:split-descendance}.
\end{proof}

In view of \Cref{lem:projection-compatible}, the split closure of the simple free augmentation is hence never worse, but could provide a potentially tighter relaxation by additional ``projected split inequalities''. We show now that this is not the case.

\begin{lemma}
\label{lem:fourier-motzkin}
    Let $\varphi: \mathcal P' \to \mathcal P$ denote the natural projection as in \Cref{lem:projection-compatible}. Then $\varphi(\mathcal P'_\mathrm{split}) = \mathcal P_\mathrm{split}$.
\end{lemma}
\begin{proof}
    Since $I'$ is an augmentation of $I$ by $\overline{a}$, we note at first, using the interpretation of split inequalities as flip inequalities from \Cref{thm:split-is-flip}, that the set of defining inequalities of $\mathcal P'_\mathrm{split}$ can be partitioned into the set of defining inequalities of $\mathcal P_\mathrm{split}$, which cannot contain the variable $x_{\overline{a}}$, and a remaining set of inequalities, which do all contain $x_{\overline{a}}$. 
    The image of $\varphi(\mathcal P'_\mathrm{split})$ can be described by Fourier-Motzkin elimination of the variable $x_{\overline{a}}$. It is therefore sufficient to show that all inequalities generated by the Fourier-Motzkin procedure are redundant for $\mathcal P_\mathrm{split}$. Since the redundancy is clear for those inequalities that do not contain $x_{\overline a}$, we will hence consider only the remaining inequalities where $x_{\overline a}$ has a non-zero coefficient.
    
    Among the defining inequalities of $\mathcal P'_\mathrm{split}$, $x_{\overline{a}}$ occurs precisely in the bound inequalities $x_{\overline{a}} \geq \ell'_{\overline{a}}$ and $x_{\overline{a}} \leq u'_{\overline{a}}$, and in the flip inequalities of simple cycles containing $\overline{a}$.  
    Fourier-Motzkin considers pairs of these inequalities, one of them giving a lower bound, and the other an upper bound on $x_{\overline{a}}$. That is, the following types of pairs have to be considered:
\begin{enumerate}[(1)]
    \item $x_{\overline{a}} \geq \ell'_{\overline{a}}$ and $x_{\overline{a}} \leq u'_{\overline{a}}$,
    \item $x_{\overline{a}} \geq \ell'_{\overline{a}}$ and a flip inequality for $(\gamma, F)$ with $\overline{a} \in \gamma$ and $\overline{a} \in F$,
    \item $x_{\overline{a}} \leq u'_{\overline{a}}$ and a flip inequality for $(\gamma, F)$ with $\overline{a} \in \gamma$ and $\overline{a} \notin F$,
    \item two flip inequalities for $(\gamma, F_\gamma)$ and $(\delta, F_\delta)$ with $\overline{a} \in \gamma$, $\overline{a} \notin F_\gamma$, $\overline{a} \in \delta$, $\overline{a} \in F_\delta$.
\end{enumerate}
In all those flip inequalities, we can assume that the cycles are simple and that the parameter $\alpha$ is at least $1$. Moreover, using the symmetry in \Cref{rem:flip-special}, we can without loss of generality fix the direction of $\overline{a}$ as forward or backward, replacing $\gamma$ by $-\gamma$ if necessary. Let us proceed with Fourier-Motzkin:
\begin{enumerate}[(1)]
    \item Elimination yields $\ell'_{\overline{a}} \leq u'_{\overline{a}}$, which is trivially true.
    \item Assume that $\overline{a}$ is forward in $\gamma$. Then we can write the flip inequality \eqref{eq:flip} for $(\gamma, F)$ with $\alpha := \alpha_{\gamma, F} \geq 1$ as
    $$ \alpha (u'_{\overline{a}} - x_{\overline{a}}) + f(x) \geq \alpha (T-\alpha),$$
    where $f(x) \geq 0$ for all $(x, z) \in \mathcal P$.
    Fourier-Motzkin elimination with $x_{\overline{a}} \geq \ell'_{\overline{a}}$ yields
    $$ \alpha u'_{\overline{a}} + f(x) \geq \alpha (T-\alpha) + \alpha \ell'_{\overline{a}},$$
    or equivalently, recalling that $u'_{\overline{a}} - \ell'_{\overline{a}} = T-1$,
        $$ f(x) \geq \alpha(T - \alpha - u'_{\overline{a}} + \ell'_{\overline{a}}) =  \alpha(1 - \alpha),$$
    but this is redundant for $\mathcal P_\mathrm{split}$, since $(x, z) \in \mathcal P$ and $\alpha \geq 1$ imply $f(x) \geq 0 \geq \alpha(1-\alpha)$.
    \item is analogous to (2).
    \item This is the most tedious part. We assume without loss of generality that $\overline{a}$ is backward in $\gamma$ and forward in $\delta$. We will show that the Fourier-Motzkin inequality is valid for all points $(x,z) \in \mathcal P$ with $(\gamma + \delta)^\top x \in T \mathbb Z$. Since $\gamma + \delta$ is an element of the cycle space $\mathcal C$ of $G$, the Fourier-Motzkin inequality is hence valid for the convex hull of those points and in particular for $\mathcal P_\mathrm{split}$ by virtue of \eqref{eq:split-def}.
    
    We first write down the flip inequalities, omitting $F_\gamma$ and $F_\delta$ in the subscripts of $\alpha$:
    \begin{align*}
        \alpha_{\gamma} (x_{\overline{a}} - \ell'_{\overline{a}}) + f(x) &\geq \alpha_\gamma (T- \alpha_\gamma),\\
        \alpha_{\delta} (u'_{\overline{a}} - x_{\overline{a}}) + g(x) &\geq \alpha_\delta (T- \alpha_\delta),
    \end{align*}
    where $f(x), g(x) \geq 0$ for all $(x, z) \in \mathcal P$.
    Elimination produces
    \begin{equation}
    \label{eq:fourier-motzkin-flip}
        \alpha_\delta f(x) + \alpha_\gamma g(x) \geq \alpha_\gamma \alpha_\delta (2T - \alpha_\gamma - \alpha_\delta - u'_{\overline{a}} + \ell'_{\overline{a}}) =  \alpha_\gamma \alpha_\delta (T + 1 - \alpha_\gamma - \alpha_\delta) .
    \end{equation}
    The inequality \eqref{eq:fourier-motzkin-flip} is trivially redundant if
    $\alpha_\gamma + \alpha_\delta \geq T+1$. We hence assume from now on $\alpha_\gamma + \alpha_\delta \leq T$.   
    Let $(x, z) \in \mathcal P$ with $(\gamma + \delta)^\top x \in T \mathbb Z$. Then
    \begin{align*}
        & & \sum_{a \in A\setminus F_\gamma} \gamma_a x_a + \sum_{a \in F_\gamma} \gamma_a x_a + \sum_{a \in A\setminus F_\delta} \delta_a x_a + \sum_{a \in F_\delta} \delta_a x_a &\equiv 0 \mod T.
    \end{align*}
    Since $\gamma_{\overline{a}} + \delta_{\overline{a}} = 0$, ${\overline{a}} \notin F_\gamma$, ${\overline{a}} \in F_\delta$, this implies
    \begin{align*}
       & & \sum_{a \in A\setminus (F_\gamma \cup \{{\overline{a}}\})} \gamma_a x_a + \sum_{a \in F_\gamma} \gamma_a x_a + \sum_{a \in A\setminus F_\delta} \delta_a x_a + \sum_{a \in F_\delta \setminus\{{\overline{a}}\}} \delta_a x_a &\equiv 0 \mod T,
    \end{align*}
    so that, using the definition of $\alpha_\gamma, \alpha_\delta$ (cf.\ \Cref{thm:flip-ineq}),
    \begin{align*}
        &\sum_{a \in A\setminus  (F_\gamma \cup \{{\overline{a}}\})} \gamma_a (x_a-\ell_{a}) - \sum_{a \in F_\gamma} \gamma_a (u_{a}-x_a) + \sum_{a \in A\setminus F_\delta} \delta_a (x_a-\ell_{a}) - \sum_{a \in F_\delta \setminus \{{\overline{a}}\}} \delta_a (u_{a}-x_a)\\
        &\equiv - \sum_{a \in A \setminus  (F_\gamma \cup \{{\overline{a}}\})} \gamma_a \ell_{a} - \sum_{a \in F_\gamma} \gamma_a u_{a} - \sum_{a \in A \setminus F_\delta} \delta_a \ell_{a} - \sum_{a \in F_\delta \setminus \{{\overline{a}}\}} \delta_a u_{a}  \mod T \\
        &\equiv \alpha_\gamma + \alpha_\delta + u'_{\overline{a}} - \ell'_{\overline{a}} \mod T\\
        &\equiv \alpha_\gamma + \alpha_\delta - 1 \mod T
    \end{align*}
    As we can assume $\alpha_\gamma, \alpha_\delta \geq 1$, we have that $\alpha := \alpha_\gamma + \alpha_\delta - 1 \geq 0$. This implies that
    \begin{align*}
        D := \sum_{a \in A\setminus  (F_\gamma \cup \{{\overline{a}}\})} \gamma_a (x_{a}-\ell_{a}) - \sum_{a \in F_\gamma} \gamma_a (u_{a}-x_{a}) + \sum_{a \in A\setminus F_\delta} \delta_a (x_{a}-\ell_{a}) - \sum_{a \in F_\delta \setminus \{{\overline{a}}\}} \delta_a (u_{a}-x_{a})
    \end{align*}
    is either $\leq \alpha - T$ (a) or $\geq \alpha$ (b).
    Before showing that \eqref{eq:fourier-motzkin-flip} is redundant in both cases, we write down the left-hand side of \eqref{eq:fourier-motzkin-flip} explicitly:
    \begin{equation}
    \label{eq:the-long-thing}
        \begin{aligned}
            &\alpha_\delta f(x) + \alpha_\gamma g(x) \\
            &= 
            \alpha_\delta (T-\alpha_\gamma) \sum_{a \in A\setminus  (F_\gamma \cup \{\overline{a}\}), \gamma_a = 1} (x_{a}-\ell_{a})
            + \alpha_\delta \alpha_\gamma \sum_{a \in A\setminus  (F_\gamma \cup \{\overline{a}\}), \gamma_a = -1} (x_{a}-\ell_{a})\\
            &
            + \alpha_\delta\alpha_\gamma \sum_{a \in F_\gamma, \gamma_a = 1} (u_{a}-x_{a})
            + \alpha_\delta (T-\alpha_\gamma) \sum_{a \in F_\gamma, \gamma_a = -1} (u_{a}-x_{a})\\
            &+ \alpha_\gamma (T-\alpha_\delta) \sum_{a \in A\setminus  F_\delta, \delta_a = 1} (x_{a}-\ell_{a})
            + \alpha_\gamma \alpha_\delta \sum_{a \in A\setminus F_\delta, \delta_a = -1} (x_{a}-\ell_{a})\\
            &
            + \alpha_\gamma\alpha_\delta \sum_{a \in F_\delta \setminus \{\overline{a}\}, \delta_a = 1} (u_{a}-x_{a})
            + \alpha_\gamma (T-\alpha_\delta) \sum_{a \in F_\delta \setminus \{\overline{a}\}, \delta_a = -1} (u_{a}-x_{a}).
        \end{aligned}
        \end{equation}
    \begin{enumerate}[(a)]
        \item Expanding $\alpha_\delta(T- \alpha_\gamma)$ and $\alpha_\gamma(T-\alpha_\delta)$ in \eqref{eq:the-long-thing}, and then bounding all summands with $T$ as a factor by $0$ from below, we obtain
        \begin{align*}
            \alpha_\delta f(x) + \alpha_\gamma g(x) 
            \geq - \alpha_\gamma \alpha_\delta D
            \geq -\alpha_\gamma \alpha_\delta(\alpha - T)
            = \alpha\gamma \alpha_\delta(T+1-\alpha_\gamma - \alpha_\delta).
            \end{align*}
            Hence \eqref{eq:fourier-motzkin-flip} holds in the case that $D \leq \alpha - T$.
        \item Let $\nu := \min(\alpha_\gamma, \alpha_\delta)$. Expanding $\alpha_\delta(T- \alpha_\gamma)$ and $\alpha_\gamma(T-\alpha_\delta)$ in \eqref{eq:the-long-thing}, bounding $T\alpha_\gamma, T\alpha_\delta$ from below by $T\nu$, we find
        \begin{align*}
            \alpha_\delta f(x) + \alpha_\gamma g(x)  \geq (T\nu 
            - \alpha_\gamma \alpha_\delta) D.
        \end{align*}
        Since $\nu$ is one of $\alpha_\gamma, \alpha_\delta$ and $\alpha_\gamma, \alpha_\delta \leq T-1$, we have $T\nu - \alpha_\gamma \alpha_\delta \geq 0$. This implies with $D \geq \alpha$ that
        \begin{align*}
        &  \alpha_\delta f(x) + \alpha_\gamma g(x) \\
            &\geq (T\nu - \alpha_\gamma \alpha_\delta) \alpha\\
            &= (T\nu - \alpha_\gamma \alpha_\delta) (\alpha_\gamma + \alpha_\delta - 1)\\
            &= T\nu (\alpha_\gamma + \alpha_\delta - 1) + \alpha_\gamma \alpha_\delta (1 - \alpha_\gamma - \alpha_\gamma) \\
            &= T\nu (\alpha_\gamma + \alpha_\delta - 1) - T\alpha_\gamma \alpha_\delta + \alpha_\gamma \alpha_\delta (T + 1 - \alpha_\gamma - \alpha_\gamma)
        \end{align*}
        It remains to show that 
        $ T\nu (\alpha_\gamma + \alpha_\delta - 1) - T\alpha_\gamma \alpha_\delta \geq 0$. 
        This is true since $\nu$ is one of $\alpha_\gamma, \alpha_\delta \geq 1$.
    \end{enumerate}
    \end{enumerate}
We conclude that the image $\varphi(\mathcal P'_\mathrm{split})$ is fully described by the flip inequalities of cycles not containing $\overline{a}$ and the variable bounds for all arcs except $\overline{a}$. Hence $\varphi(\mathcal P'_\mathrm{split}) = \mathcal P_\mathrm{split}$.
\end{proof}

As a corollary to \Cref{lem:fourier-motzkin}, we obtain the following result.
\begin{theorem}
\label{thm:free-augmentation}
    Let $I$ and $I'$ be PESP instances with fractional periodic timetabling polyhedra $\mathcal P$ and $\mathcal P'$, respectively. Suppose that $I'$ arises from $I$ by a sequence of simple free augmentations. If $\varphi: \mathcal P' \to \mathcal P$ denotes the natural projection, then $\varphi(\mathcal P'_\mathrm{split}) = \mathcal P_\mathrm{split}$.
\end{theorem}
In particular, recalling \Cref{rem:incidence-is-cycle}, the incidence-based formulation \eqref{eq:mip-incidence} is not stronger than the cycle-based formulation \eqref{eq:mip-cycle} in terms of split closures.
Consequently, it is of no use to develop methods which augment an instance by a free arc, obtain a flip/split inequality and project down again, as this will not lead to information which cannot already be obtained from the split closure of the original instance.

\subsection{Binarization by Subdivision}
\label{ss:binarization}

A reformulation of a MIP general variables into one with binary variables can exhibit stronger split closures \citep{dash_binary_2018} or lift-and-project closures \citep{aprile_binary_2021}. For the application of periodic timetabling, there is a combinatorial binarization method: Let $I = (G, T, \ell, u, w)$ be a PESP instance, $G = (V, A)$. We assume that $0 \leq \ell < T$ and $\ell \leq u < \ell + T$ by preprocessing (see \Cref{rem:preprocessing}), so that the integer periodic offset variables $p_a$ in the incidence-based formulation \eqref{eq:mip-incidence} of PESP can only take values in $\{0, 1, 2\}$.
Moreover, if $u_a \leq T$ for some $a \in A$, then $p_a \in \{0, 1\}$ for any integer feasible solution $(x, \pi, p)$ \citep{liebchen_periodic_2006}.

\begin{definition}
Let $I' = (G', T, \ell', u', w')$ be a PESP instance that arises from $I$ by subdividing an arc $\overline{a} \in A$ with $\ell_{\overline{a}} < u_{\overline{a}}$ into two new arcs $a_1, a_2$ such that:
\begin{align*}
0\leq \ell'_{a_1} &\leq u'_{a_1},\\
0 \leq \ell'_{a_2} &\leq u'_{a_2},\\
\ell'_{a_1}+ \ell'_{a_2} &= \ell_{\overline{a}},\\
u'_{a_1}+ u'_{a_2} &= u_{\overline{a}},\\
w'_{a_1} = w'_{a_2} &= w_{\overline{a}}.
\end{align*}
We call $I'$ a \emph{simple subdivision} of $I$ at $\overline{a}$. 
\end{definition}

Observe that if the bounds on the arc $\overline{a}$ are such that $u_{\overline{a}} > T$, one can always construct a simple subdivision $I'$ of $I$ at $\overline{a}$ such that $u'_{a_i} - \ell'_{a_i}> 0$ and $u'_{a_i} \leq T$ for $i\in \{1,2\}$, due to the assumption that $\ell <T$ and $u< \ell+T$. 
%
%
As a result, for the instance $I'$ arising from subdividing each arc $\overline{a}$ with $u_{\overline{a}} > T$ as above, the incidence-based MIP formulation \eqref{eq:mip-incidence} will then have exclusively binary variables.

\begin{example}
\Cref{fig:pesp-subdivision} shows the instance obtained from the instance $I$ from \Cref{fig:pesp-example} by subdividing every arc with $u_a > T$. 

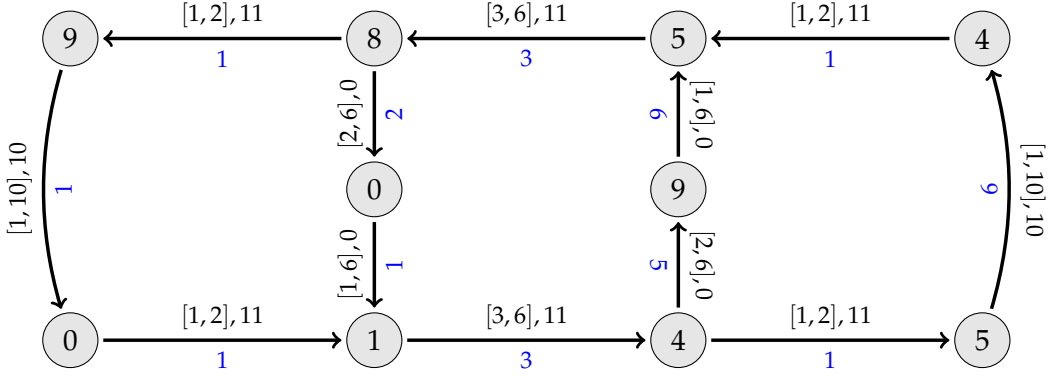
\begin{figure}[htbp]
    \centering
    \def\circledarrow#1#2#3{ 
	\draw[#1,<-] (#2) +(60:#3) arc(60:-260:#3);
}

\begin{tikzpicture}[yscale=2]
	\definecolor{highlight}{HTML}{4eed98}
	\definecolor{mint}{HTML}{14b861}
	\definecolor{darkblue}{HTML}{003399}
	\tikzstyle{a} = [line width=1.3, ->]
	\tikzstyle{am} = [line width=6, highlight]
	\tikzstyle{t} = [midway, font=\footnotesize]
	\tikzstyle{ta} = [t, above]
	\tikzstyle{tb} = [t, below, blue]
	\tikzstyle{tr} = [t, right]
	\tikzstyle{tl} = [t, left]
	\tikzstyle{v} = [draw, circle, inner sep=4, fill = gray!20, outer sep = 2]
	\tikzstyle{c} = [gray]
	\tikzmath{\h = 4; \v = 2;};
	\node[v] (A1) at (0,0) {0};
	\node[v] (A2) at ($(1*\h,0)$) {1};
	\node[v] (A3) at ($(2*\h,0)$) {4};
	\node[v] (A4) at ($(3*\h,0)$) {5};
	
	\node[v] (B1) at ($(0*\h,\v)$) {9};
	\node[v] (B2) at ($(1*\h,\v)$) {8};
	\node[v] (B3) at ($(2*\h,\v)$) {5};
	\node[v] (B4) at ($(3*\h,\v)$) {4};

	\node[v] (C2) at ($(1*\h,0.5*\v)$) {0};
	\node[v] (C3) at ($(2*\h,0.5*\v)$) {9};
	
	\draw[a] (A1) -- node[ta] {$[1,2],11$}  node[tb] {$1$} (A2);
	\draw[a] (A2) -- node[ta] {$[3,6],11$} node[tb] {$3$} (A3);
	\draw[a] (A3) -- node[ta] {$[1,2],11$} node[tb] {$1$} (A4);
	
	\draw[a] (B4) -- node[ta] {$[1,2],11$} node[tb] {$1$} (B3);
	\draw[a] (B3) -- node[ta] {$[3,6],11$} node[tb] {$3$} (B2);
	\draw[a] (B2) -- node[ta] {$[1,2],11$} node[tb] {$1$} (B1);
	
	\draw[a] (A4) to[bend right] node[ta, rotate = -90] {$[1,10],10$} node[tb, rotate = -90] {$9$} (B4);
	\draw[a] (B1)  to[bend right] node[ta, rotate=90] {$[1,10],10$} node[tb, rotate = 90] {$1$} (A1);
	\draw[a] (B2) -- node[ta, rotate = 90] {$[2,6],0$}  node[tb, rotate = 90] {$2$}  (C2);
	\draw[a] (C2) -- node[ta, rotate = 90] {$[1,6],0$} node[tb, rotate = 90] {$1$} (A2);
	\draw[a] (A3) -- node[ta, rotate = -90] {$[2,6],0$} node[tb, rotate = -90] {$5$} (C3);
	\draw[a] (C3) -- node[ta, rotate = -90] {$[1,6],0$} node[tb, rotate = -90] {$6$} (B3);

\end{tikzpicture}
    \caption{Subdivision of the instance in \Cref{fig:pesp-example} obtained from two simple subdivisions such that $u_a \leq T$ for all arcs $a$. }

    \label{fig:pesp-subdivision}
\end{figure}

\end{example}

Let $I'$ be a simple subdivision of a PESP instance $I$ at $\overline{a}$, introducing new arcs $a_1$ and $a_2$. The cycle spaces of $G$ and $G'$ are isomorphic: If $\gamma$ is an element of the cycle space of $G$, then $\gamma'$ with
$$
\gamma'_a := \begin{cases}
    \gamma_{\overline{a}} & \text{ if } a \in \{a_1, a_2\}, \\
    \gamma_a & \text{ if } a \notin\{a_1, a_2\},
\end{cases}
$$
defines an element of the cycle space of $G'$, and the whole cycle space of $G'$ arises this way. We can therefore associate to an integral cycle basis $B = \{\gamma_1, \dots, \gamma_\mu\}$ of $G$ the integral cycle basis $B' = \{\gamma_1', \dots, \gamma_\mu'\}$. Then any cycle offset $z$ in \eqref{eq:mip-cycle} w.r.t.\ $B$ defines a cycle offset $z'$ w.r.t.\ $B'$ by $z'_{\gamma'} := z_\gamma$, so that cycle offsets are essentially the same.
We will use $\Gamma$ and $\Gamma'$ to define $\mathcal P$ and $\mathcal P'$, the fractional periodic tension polytopes of $I'$ and $I$, respectively.

\begin{lemma}
\label{lem:subdivision-compatible}
Consider a simple subdivision $I'$ of $I$ at an arc $\overline{a}$ with 
notation as above.
\begin{enumerate}[\normalfont (1)]
    \item The map $\rho: \mathcal P' \to \mathcal P, (x, x_{a_1}, x_{a_2}, z) \mapsto (x, x_{a_1} + x_{a_2}, z)$ is well-defined and mixed-integer-compatible.
    \item The map $s: \mathcal P \to \mathcal P', (x, x_{\overline{a}}, z) \mapsto \left(x, \ell'_{a_1} + \frac{u'_{a_1} - \ell'_{a_1}}{u_{\overline{a}} - \ell_{\overline{a}}} (x_{\overline{a}} - \ell_{\overline{a}}), \ell'_{a_2} + \frac{u'_{a_2} - \ell'_{a_2}}{u_{\overline{a}} - \ell_{\overline{a}}} (x_{\overline{a}} - \ell_{\overline{a}}), z \right)$ is well-defined and mixed-integer-compatible.
    \item $\rho \circ s: \mathcal P \to \mathcal P$ is the identity map.
    \item $\rho(\mathcal P'_{\mathrm split}) = \mathcal P_\mathrm{split}$.
\end{enumerate}
\end{lemma}
\begin{proof}
\begin{enumerate}[(1)]
    \item The map is well-defined: The hypothesis $\ell'_{a_1} + \ell'_{a_2} = \ell_{\overline{a}}$ and $u'_{a_1} + u'_{a_2} = u_{\overline{a}}$ implies that $\ell_{\overline{a}} \leq x_{a_1}+x_{a_2} \leq u_{\overline{a}}$ holds for all $(x, x_{a_1}, x_{a_2}, z) \in \mathcal P'$. As $\rho$ is linear and does not affect the integrality of $z$, it is mixed-integer-compatible.
    \item The map is well defined: 
    Due to the assumption of subdividing arcs with $u_{\overline{a}} > T$ only, we have
    $u_{\overline{a}} - \ell_{\overline{a}} > 0$ and $u'_{a_1} - \ell'_{a_1} \geq 0$. Since $x_{\overline{a}} - \ell_{\overline{a}} \geq 0$ for all $(x, x_{\overline{a}}, z) \in \mathcal P$, we conclude
    $$ \ell'_{a_1} = \ell'_{a_1} + \frac{u'_{a_1} - \ell'_{a_1}}{u_{\overline{a}} - \ell_{\overline{a}}} (\ell_{\overline{a}} - \ell_{\overline{a}})  \leq \ell'_{a_1} + \frac{u'_{a_1} - \ell'_{a_1}}{u_{\overline{a}} - \ell_{\overline{a}}} (x_{\overline{a}} - \ell_{\overline{a}}) \leq \ell'_{a_1} + \frac{u'_{a_1} - \ell'_{a_1}}{u_{\overline{a}} - \ell_{\overline{a}}} (u_{\overline{a}} - \ell_{\overline{a}}) = u'_{a_1}. $$
    The argument for the $x_{a_2}$ entry is analogous. We note that $s$ is affine and maps point with integral $z$ to points with integral $z$, so that $s$ is mixed-integer-compatible.
    \item This follows since
    $$ \ell'_{a_1} + \frac{u'_{a_1} - \ell'_{a_1}}{u_{\overline{a}} - \ell_{\overline{a}}} (x_{\overline{a}} - \ell_{\overline{a}})
    + \ell'_{a_2} + \frac{u'_{a_2} - \ell'_{a_2}}{u_{\overline{a}} - \ell_{\overline{a}}} (x_{\overline{a}} - \ell_{\overline{a}})
    = \ell_{\overline{a}} + \frac{u_{\overline{a}} - \ell_{\overline{a}}}{u_{\overline{a}} - \ell_{\overline{a}}} (x_{\overline{a}} - \ell_{\overline{a}}) = x_{\overline{a}}.$$
    \item Since $\rho$ and $s$ are mixed-integer compatible, $\rho(\mathcal P'_\mathrm{split}) \subseteq \mathcal P_\mathrm{split}$ and $s(\mathcal P_\mathrm{split}) \subseteq \mathcal P'_\mathrm{split}$. The composition $\rho|_{\mathcal P'_\mathrm{split}} \circ s|_{\mathcal P_\mathrm{split}}$ of the restrictions to split closures is hence well-defined, and by (3), it is the identity map on $\mathcal P_\mathrm{split}$. We conclude that $\rho|_{\mathcal P'_\mathrm{split}} $ is surjective. \qedhere
\end{enumerate}
\end{proof}

A repeated application of \Cref{lem:subdivision-compatible} together with \Cref{thm:free-augmentation} yields:

\begin{theorem}
\label{thm:binarization}
    Let $I$ and $I'$ be PESP instances with fractional periodic timetabling polyhedra $\mathcal P$ and $\mathcal P'$, respectively. Suppose that $I'$ arises from $I$ by a sequence of simple subdivisions and simple free augmentations. If $\psi: \mathcal P' \to \mathcal P$ denotes the composition of the summation maps $\rho$ in \Cref{lem:subdivision-compatible}~(1) for the subdivisions and the projection maps $\varphi$ in \Cref{lem:projection-compatible} for the free augmentations, then $\psi(\mathcal P'_\mathrm{split}) = \mathcal P_\mathrm{split}$.
\end{theorem}

In particular, when we binarize the MIP \eqref{eq:mip-cycle} by first performing simple subdivisions and then move to the formulation \eqref{eq:mip-incidence}, we gain no further insight about split inequalities.

\section{Computational Experiments}
\label{sec:experiments}

We want to assess how useful the split closure is for obtaining dual bounds for PESP in practice. To that end we introduce a procedure, which exploits \Cref{thm:sepa-fixed-cycle} in a heuristic way, and proceeds to find cuts  systematically once the heuristic fails, such that we optimize over the entire split closure by means of \Cref{thm:sepa-parametric-mip}. We will also examine the performance of the heuristic in comparison to the systematic exploration.

\subsection{Separation Procedure}

Our goal is to optimize over the entire split closure. We do so with our custom separator which proceeds as illustrated by the flowchart in \Cref{fig:flowchart}: At first, it tries to heuristically generate violated flip inequalities (highlighted in blue in the chart): We compute a minimum spanning tree with respect to the periodic slack $x - \ell$ of the current LP solution $(x, z) \in \mathcal P$, and determine a most violated flip inequality for the fundamental cycles of that tree by \Cref{thm:sepa-fixed-cycle}. When no more heuristic cuts are found, the parametric IP \eqref{eq:pmip} as in \Cref{thm:sepa-parametric-mip} is solved. During the solution process of the IP, a callback retrieves intermediate incumbent solutions and generates the corresponding cuts. The procedure terminates when no more violated cuts can be found, or the time limit is hit. Since the amount of cuts found by the heuristic is rather larger in the beginning, we apply the filtering mechanisms of SCIP to detect effective cuts. However, cuts found by the parametric IP will always be enforced, so that the whole procedure is correct up to numerical tolerances: If the procedure terminates because no more violated cuts can be detected, then the optimal solution over the split closure has been found.

\begin{figure}[htbp]
    \centering
    \scalebox{0.95}{\begin{tikzpicture}[scale=1, node distance=2cm]
\definecolor{myblue}{cmyk}{.98, .15, 0, .47}
\tikzstyle{process} = [rectangle, minimum width=3cm, minimum height=1cm, text centered, draw=black, fill=orange!30, align=center]
\tikzstyle{decision} = [diamond, minimum width=3cm, minimum height=1cm, text centered, draw=black, fill=green!30, align=center, aspect=2, inner sep=-5]
\tikzstyle{arrow} = [thick,->,>=stealth]
\tikzstyle{heuristic} = [myblue]
\tikzstyle{startstop} = [rectangle, rounded corners, minimum width=3cm, minimum height=1cm,text centered, draw=black, fill=red!30, align=center]

\node (sta1) [startstop] {Initialize the LP. Set $\alpha = \lfloor T/2 \rfloor$.};
\node (pro1) [process, below of=sta1] {Solve LP.};
\node (dec1) [decision, below of=pro1, inner sep=0] {Is $z$ integral?};
\node (pro2) [process, below of=dec1] {Compute minimum spanning\\tree of $G$ w.r.t.\ $x - \ell$.};
\node (pro3) [process, below of=pro2] {Compute maximally violated flip inequality \\ w.r.t.\ $\gamma_1, \dots, \gamma_\mu$ and $(x, z)$ via \Cref{thm:sepa-fixed-cycle}.};
\node (dec2) [decision, below of=pro3, yshift=-0.5cm] {Are there\\flip inequalities that\\cut off $(x, z)$?};
\node (pro4) [process, right of=dec2, xshift=3cm] {Add violated flip\\inequalities to LP.};
\node (pro5) [process, below of=dec2, yshift=-0.5cm] {Solve IP \eqref{eq:pmip} w.r.t.\ $\alpha$ and $(x, z)$.};
\node (dec3) [decision, below of=pro5, yshift=-0.5cm] {Are there\\flip inequalities that\\cut off $(x, z)$?};
\node (pro6) [process, right of=dec3, xshift=5cm] {Add violated flip\\inequalities to LP.\\Decrease $\alpha$.};
\node (pro7) [process, below of=dec3, yshift=-0.5cm] {Decrease $\alpha$.};
\node (dec4) [decision, below of=pro7, yshift=-0.5cm] {Has IP \eqref{eq:pmip}\\for $\alpha$ and $(x, z)$ \\ been considered\\before?};
\node (sto1) [startstop, right of=dec4, xshift=5cm] {Stop.};
\node (sto2) [startstop, left of=dec1, xshift=-3cm] {Stop.};
\node (dum1) [left of=dec4, xshift=-3cm, inner sep=0] {};

\draw [arrow] (sta1) -- node[right] {LP relaxation of \eqref{eq:mip-cycle}} (pro1);
\draw [arrow, heuristic] (pro1) -- node[right] {solution $(x, z) \in \mathcal P$} (dec1);
\draw [arrow, heuristic] (dec1) -- node[right] {no} (pro2);
\draw [arrow, heuristic] (pro2) -- node[right, align=left] {fund.\ cycles $\gamma_1, \dots, \gamma_\mu$} (pro3);
\draw [arrow, heuristic] (pro3) --  node[right] {flip inequalities} (dec2);
\draw [arrow, heuristic] (dec2) --  node[above] {yes} (pro4);

\draw [arrow] (dec2) --  node[right] {no} (pro5);
\draw [arrow] (pro5) --  node[right] {flip inequalities} (dec3);
\draw [arrow] (dec3) --  node[above] {yes} (pro6);
\draw [arrow] (pro6) |-  (pro1);
\draw [arrow] (dec3) --  node[right] {no} (pro7);
\draw [arrow] (pro7) --  node[right] {} (dec4);
\draw [arrow] (dec4) -- node[above] {no}  (dum1) |-  (pro5);
\draw [arrow] (dec4) --  node[above] {yes} (sto1);
\draw [arrow] (dec1) --  node[above] {yes} (sto2);
\draw [arrow, heuristic] (pro4) |-  (pro1);	
\end{tikzpicture}}
    \caption{Flowchart of the split cut generation procedure. ``Decrease $\alpha$'' means to set $\alpha := \alpha - 1$ if $\alpha \geq 2$ and $\alpha := \lfloor T/2 \rfloor$ otherwise. We start with $\alpha = \lfloor T/2 \rfloor$, as this is likely to produce cuts with large violation.} 
    \label{fig:flowchart}
\end{figure}
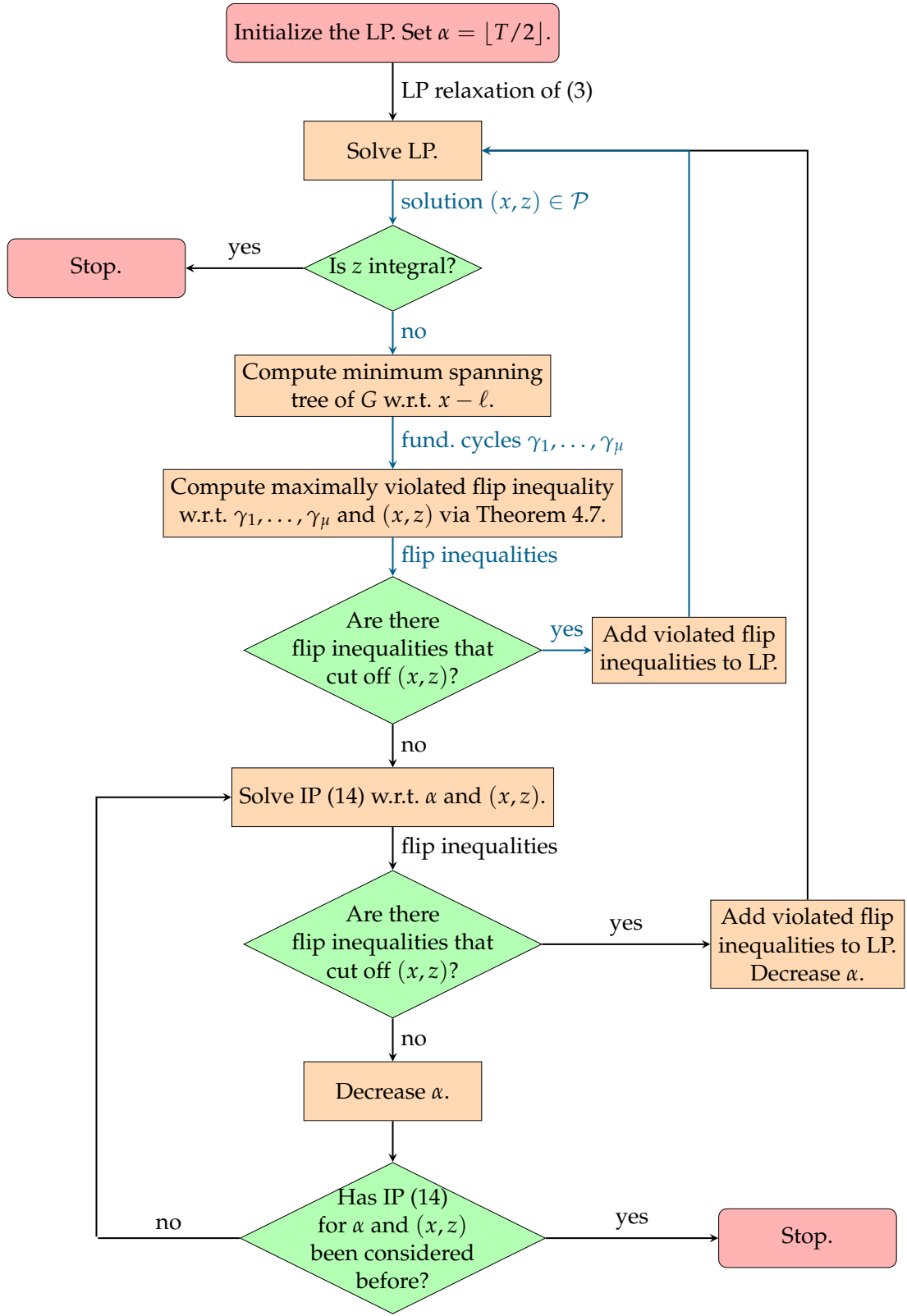

\subsection{Methodology}
\label{sec:methodology}

To conduct our computational experiments, we use the benchmark library PESPlib \citep{PESPlib}, whose instances are derived from real-world scenarios. Although significant process has been made in the past, no instance could be solved to proven optimality up to date.

Since the PESPlib instances are computationally very hard, we consider not only the full instances, but also two subinstances per instance whose cyclomatic number $\mu$ has been restricted to 25, 
and 100, respectively. Note that $\mu$ is the number of integer variables in \eqref{eq:mip-cycle}. The restriction procedure for an instance $(G, T, \ell, u, w)$ works by iteratively removing arcs, deleting in each step one arc $a$ with highest span $u_a - \ell_a$ and breaking ties by preferring lowest weight $w_a$ (cf.\ \citealt{goerigk_improved_2017}). In contrast to the full PESPlib instances, these restricted variants can be solved to optimality within a reasonable amount of time.

We first preprocess each instance, so that in particular the assumptions as in \Cref{rem:preprocessing} hold. For each instance $I = (G, T, \ell, u, w)$, we consider the cycle-based formulation \eqref{eq:mip-cycle} using an integral cycle basis $B$ that minimizes $\sum_{\gamma \in B} \sum_{a \in \gamma} (u_a - \ell_a)$. This choice of cycle basis is motivated by its good performance for computing dual bounds  \citep{borndorfer_concurrent_2020,masing_forward_2023}. By \Cref{thm:free-augmentation}, the cycle-based formulation is not weaker than the incidence-based formulation, and by \Cref{rem:incidence-is-cycle}, it is more compact. We then invoke the branch-cut-and-price framework SCIP \citep{achterberg_scip_2009}. The advantage of using SCIP is that it is highly customizable and we can disable everything that does not come from split cuts: We disable the built-in presolving, branching, heuristics, propagators separators and merely call a custom separation callback during the cutting loop at the root node.

In all experiments, we use SCIP 8.0.3 \citep{bestuzheva_scip_2021} with Gurobi 9.5.2 \citep{gurobi} as LP solver. We also use Gurobi to solve the parametric IP from \Cref{thm:sepa-parametric-mip}. Gurobi is allowed to use 6 threads on an Intel Xeon E3-1270 v6 CPU running at 3.8 GHz with 32 GB RAM. The time limit has been set to 4 hours wall time for each instance.

\subsection{Results}
\label{sec:results}

\subsubsection{Restriction to $\mu = 25$}

\Cref{tab:results-mu-25} shows the results for the restrictions of the PESPlib instances to the cyclomatic number $\mu = 25$. For all but one instance, the cut generation procedure of \Cref{sec:methodology} terminates within 22 minutes, only R1L1v hits the time limit due to a hard parametric IP \eqref{eq:pmip}.  Optimizing over the split closure is exact for R4L1 and R4L4v, but R4L4v is trivial in the sense that $x = \ell$ is an optimal solution. The average relative optimality gap with respect to the optimal objective value in terms of weighted slack $w^\top(x-\ell)$ and the best bound obtained by split cuts, taken over all 22 instances, is 6.61\,\%.

\begin{table}[htbp]
\begin{center}
\begin{tabular}{lrrrrrrr}
Instance & $\mu$ & Opt.\ Val.\ ($\mathcal P_\mathrm{I}$) & Dual Bd.\ ($\mathcal P_\mathrm{split}$) & Gap [\%] & Cuts & IP Cuts & Time [s] \\
\hline
BL1   &     25 &     479\,501 &     455\,492 &  5.01 &     114 &  33 &      98 \\
BL2   &     25 &     582\,203 &     529\,247 &  9.10 &     128 &  32 &     116 \\
BL3   &     25 &     614\,544 &     513\,344 & 16.47 &     122 &  28 &     197 \\
BL4   &     25 &     581\,688 &     507\,168 & 12.81 &     176 &  65 &     106 \\
\hline
R1L1  &     25 &  1\,469\,763 &  1\,314\,105 & 10.59 &     284 & 123 &     747 \\
R1L2  &     25 &  1\,271\,066 &  1\,235\,774 &  2.78 &     226 &  96 &     857 \\
R1L3  &     25 &  1\,704\,349 &  1\,693\,441 &  0.64 &     238 & 114 &  1\,281 \\
R1L4  &     25 &  1\,543\,182 &  1\,429\,795 &  7.35 &     294 & 118 &     936 \\
\hline
R2L1  &     25 &  2\,598\,725 &  2\,171\,855 & 16.43 &     212 &  83 &     255 \\
R2L2  &     25 &  2\,726\,109 &  2\,471\,181 &  9.35 &     238 &  75 &     335 \\
R2L3  &     25 &  1\,698\,794 &  1\,661\,074 &  2.22 &     116 &  12 &      91 \\
R2L4  &     25 &  2\,417\,447 &  2\,325\,110 &  3.82 &     244 &  64 &     119 \\
\hline
R3L1  &     25 &  1\,110\,721 &  1\,055\,499 &  4.97 &     170 &  83 &     513 \\
R3L2  &     25 &  1\,283\,884 &  1\,148\,551 & 10.54 &     152 &  67 &     201 \\
R3L3  &     25 &  1\,617\,501 &  1\,478\,034 &  8.62 &     196 &  58 &     389 \\
R3L4  &     25 &  1\,063\,438 &     987\,067 &  7.18 &     143 &  56 &     399 \\
\hline
R4L1  &     25 &  1\,053\,623 &  1\,053\,623 &  0.00 &     102 &  14 &     205 \\
R4L2  &     25 &  1\,394\,526 &  1\,313\,700 &  5.80 &     136 &  39 &     231 \\
R4L3  &     25 &  1\,718\,591 &  1\,648\,388 &  4.08 &     148 &  47 &     213 \\
R4L4  &     25 &     498\,913 &     488\,043 &  2.18 &     171 &  60 &     701 \\
\hline
R1L1v &     25 &  1\,741\,592 &  1\,645\,779 &  5.50 &     128 &  58 & 14\,400 \\
R4L4v &     25 &  3\,660\,000 &  3\,660\,000 &  0.00 &       0 &   0 &       0 \\
\end{tabular}
\caption{Results for the PESPlib instances restricted to $\mu = 25$. The table lists the optimal objective value of the MIP \eqref{eq:mip-cycle} in terms of weighted slack $w^\top(x- \ell)$, the best dual bound obtained by split cuts, the primal-dual gap, the total number of applied split cuts, the number of cuts provided by the parametric IP \eqref{eq:pmip}, and the running time in seconds.}
\label{tab:results-mu-25}
\end{center}
\end{table}

\subsubsection{Restriction to $\mu = 100$}

The results for the restriction to $\mu = 100$ are summarized in \Cref{tab:results-mu-100}. Again, we can determine the optimal solution of \eqref{eq:mip-cycle} for all these restricted instances. The cut generation procedure of \Cref{sec:methodology} terminates within the time limit for 20 out of 22 instances. R4L4v is again almost trivial to solve, because two cuts suffice to produce an integral solution. The second smallest gap is at R1L1v, although the time limit is hit. The average optimality gap is 13.37\,\%, which is about twice as much as in the case $\mu = 25$.


\begin{table}[htbp]
\begin{center}
\begin{tabular}{lrrrrrrr}
Instance & $\mu$ & Opt.\ Val.\ ($\mathcal P_\mathrm{I}$) & Dual Bd.\ ($\mathcal P_\mathrm{split}$) & Gap [\%] & Cuts & IP Cuts & Time [s] \\
\hline
BL1   &    100 &  1\,341\,151 &  1\,216\,355 &  9.31 &  1\,092 & 357 &     954 \\
BL2   &    100 &  1\,733\,429 &  1\,451\,049 & 16.29 &     910 & 307 &  1\,287 \\
BL3   &    100 &  1\,747\,063 &  1\,461\,798 & 16.33 &     922 & 304 &  1\,864 \\
BL4   &    100 &  1\,605\,968 &  1\,427\,228 & 11.13 &     975 & 349 &  1\,399 \\
\hline
R1L1  &    100 &  5\,481\,154 &  4\,582\,018 & 16.40 &  1\,300 & 493 &  6\,903 \\
R1L2  &    100 &  4\,873\,559 &  3\,952\,695 & 18.90 &  1\,138 & 348 &  7\,453 \\
R1L3  &    100 &  6\,256\,521 &  5\,151\,095 & 17.67 &     998 & 324 &  4\,742 \\
R1L4  &    100 &  5\,008\,640 &  4\,202\,959 & 16.09 &  1\,407 & 415 &  6\,184 \\
\hline
R2L1  &    100 &  8\,284\,107 &  6\,881\,776 & 16.93 &  1\,021 & 294 &  2\,453 \\
R2L2  &    100 &  7\,099\,578 &  6\,244\,993 & 12.04 &  1\,366 & 406 &  4\,648 \\
R2L3  &    100 &  6\,722\,776 &  5\,982\,798 & 11.01 &  1\,102 & 342 &  6\,038 \\
R2L4  &    100 &  5\,516\,243 &  4\,996\,368 &  9.42 &  1\,217 & 317 &  3\,242 \\
\hline
R3L1  &    100 &  4\,366\,123 &  3\,770\,709 & 13.64 &     927 & 355 &  7\,180 \\
R3L2  &    100 &  4\,666\,798 &  3\,796\,483 & 18.65 &     764 & 253 &  8\,554 \\
R3L3  &    100 &  4\,719\,345 &  3\,890\,774 & 17.56 &     921 & 301 &  7\,492 \\
R3L4  &    100 &  2\,950\,612 &  2\,730\,898 &  7.45 &     885 & 348 & 11\,242 \\
\hline
R4L1  &    100 &  4\,428\,800 &  3\,715\,032 & 16.12 &     717 & 179 &  2\,947 \\
R4L2  &    100 &  4\,101\,438 &  3\,492\,759 & 14.84 &     789 & 236 & 14\,400 \\
R4L3  &    100 &  4\,302\,565 &  3\,740\,673 & 13.06 &     875 & 226 &  8\,785 \\
R4L4  &    100 &  1\,994\,572 &  1\,676\,547 & 15.94 &     607 & 184 &  7\,448 \\
\hline
R1L1v &    100 & 10\,253\,906 &  9\,715\,723 &  5.25 &     191 &   0 & 14\,400 \\
R4L4v &    100 & 14\,880\,000 & 14\,880\,000 &  0.00 &       2 &   0 &       1 \\
\end{tabular}
\caption{Results for the PESPlib instances restricted to $\mu = 100$. The table lists the optimal objective value of the MIP \eqref{eq:mip-cycle} in terms of weighted slack $w^\top(x- \ell)$, the best dual bound obtained by split cuts, the primal-dual gap, the total number of applied split cuts, the number of cuts provided by the parametric IP \eqref{eq:pmip}, and the running time in seconds.}
\label{tab:results-mu-100}
\end{center}
\end{table}


\subsubsection{Full instances}

Finally, the results for the full PESPlib instances are given in \Cref{tab:results-mu-full} in comparison to the best known primal bounds and in \Cref{tab:results-full-dual} in comparison to the best known dual bounds. All instances hit the time limit. Compared to the restricted instances, relatively few cuts are generated by the IP \eqref{eq:pmip}, which is both due to the large supply of heuristically generated cuts, and the difficulty of the IP. The time limit is not sufficient to unfold the power of the IP, on the other hand, increasing the time limit to 8 or 24 hours empirically produced only marginal improvements. This effect is also illustrated in \Cref{fig:R2L1-full-progress}: The plot shows an exemplary progression of the dual bound and the number of applied cuts for the instance R2L1 with a logarithmic time axis. The heuristic separation procedure finds no more cuts for the first time after roughly 45 minutes (about $10^{3.43}$ seconds), and then the parametric IP takes over, causing a sudden and persisting drop in performance. 
We can observe that once the parametric IP came into effect, the heuristic stage provides only few further cuts.
This could be due to the initial high quality results provided by the heuristic, such that the improvement through a cut from the parametric IP results in only a marginal change in the new solution. The subsequent spanning tree in the following heuristic stage could then be similar to the previous one, such that from this point on, only little to no improvement is found in the heuristic stage; and the costly parametric IP is the main contributor.

With respect to all instances, the average optimality gap is 40.83\,\%. As expected, the quality of the results obtained by our method is dependent on the problem size. In particular for the 16 R$i$L$j$ instances there is a strong correlation between the size of $\mu$ and the optimality gap. This is also evidenced by the Pearson correlation coefficient, which is approximately 95\,\%.

On the dual side, the split closure provides at least 91.10\,\% of the currently best known dual bound. This underlines the good performance of our method -- most of the incumbent dual bounds have been obtained by longer computation times, and in contrast to our study, neither branching nor other types of cutting planes apart from split cuts have been forbidden. 
Despite being at a disadvantage in this regard, our method provides better dual bounds for five out of the 22 instances, with improvements up to 25\,\%. 
Other bounds have been obtained with the help of heuristically separated flip inequalities as well, by, e.g., \cite{borndorfer_concurrent_2020,lindner_determining_2020,lindner_forward_2021,masing_forward_2023}, such that our procedure can be seen as an advancement of previous methods in the sense that our heuristic unlocks more potential due to exploiting \Cref{thm:sepa-fixed-cycle}.

\begin{table}[htbp]
\begin{center}
\begin{tabular}{lrrrrrrr}
Instance & $\mu$ & Primal Bd.\ ($\mathcal P_\mathrm{I}$) & Dual Bd.\ ($\mathcal P_\mathrm{split}$) & Gap [\%] & Cuts & IP Cuts & Time [s] \\
\hline
BL1   & 5\,298 &  6\,333\,641 &  4\,252\,778 & 32.85 & 42\,927 &  15 & 14\,400 \\
BL2   & 4\,880 &  6\,799\,331 &  4\,299\,517 & 36.77 & 37\,498 &  84 & 14\,400 \\
BL3   & 6\,265 &  6\,675\,098 &  4\,290\,946 & 35.72 & 58\,628 &  20 & 14\,400 \\
BL4   & 9\,684 &  6\,562\,147 &  3\,923\,974 & 40.20 & 88\,640 & 265 & 14\,400 \\
\hline
R1L1  & 2\,722 & 29\,894\,745 & 19\,041\,890 & 36.30 & 21\,965 &  60 & 14\,400 \\
R1L2  & 2\,876 & 30\,507\,180 & 19\,059\,669 & 37.52 & 23\,767 &  45 & 14\,400 \\
R1L3  & 2\,848 & 29\,319\,593 & 18\,193\,974 & 37.95 & 23\,468 &  61 & 14\,400 \\
R1L4  & 3\,769 & 26\,516\,727 & 16\,441\,121 & 38.00 & 30\,460 &  18 & 14\,400 \\
\hline
R2L1  & 3\,206 & 42\,422\,038 & 24\,806\,675 & 41.52 & 27\,739 & 163 & 14\,400 \\
R2L2  & 3\,360 & 40\,642\,186 & 24\,464\,467 & 39.81 & 28\,842 & 159 & 14\,400 \\
R2L3  & 3\,239 & 38\,558\,371 & 22\,645\,939 & 41.27 & 28\,816 &  95 & 14\,400 \\
R2L4  & 5\,514 & 32\,483\,894 & 19\,102\,410 & 41.19 & 47\,958 &   0 & 14\,400 \\
\hline
R3L1  & 4\,630 & 43\,271\,824 & 25\,343\,534 & 41.43 & 38\,725 &  17 & 14\,400 \\
R3L2  & 4\,800 & 45\,220\,083 & 25\,963\,773 & 42.58 & 41\,951 &  19 & 14\,400 \\
R3L3  & 5\,446 & 40\,483\,617 & 22\,273\,090 & 44.98 & 46\,099 &   6 & 14\,400 \\
R3L4  & 7\,478 & 33\,335\,852 & 17\,027\,192 & 48.92 & 46\,773 &   0 & 14\,400 \\
\hline
R4L1  & 5\,331 & 49\,426\,919 & 27\,938\,824 & 43.47 & 42\,505 &   6 & 14\,400 \\
R4L2  & 5\,688 & 48\,764\,793 & 27\,585\,028 & 43.43 & 45\,946 &   7 & 14\,400 \\
R4L3  & 6\,871 & 45\,493\,081 & 23\,849\,465 & 47.58 & 46\,277 &   0 & 14\,400 \\
R4L4  & 9\,371 & 36\,703\,391 & 16\,488\,684 & 55.08 & 42\,579 &   0 & 14\,400 \\
\hline
R1L1v & 2\,832 & 42\,591\,141 & 28\,544\,123 & 32.98 & 20\,326 &  22 & 14\,400 \\
R4L4v & 9\,637 & 61\,968\,380 & 38\,307\,814 & 38.18 & 45\,916 &   0 & 14\,400 \\
\end{tabular}
\caption{Results for the full PESPlib instances. The table lists the best known primal bound for the MIP \eqref{eq:mip-cycle} in terms of weighted slack $w^\top(x- \ell)$ according to \citep{PESPlib}, the best dual bound obtained by split cuts, the primal-dual gap, the total number of applied split cuts, the number of cuts provided by the parametric IP \eqref{eq:pmip}, and the running time in seconds.}
\label{tab:results-mu-full}
\end{center}
\end{table}

\begin{figure}
    \centering
    \includegraphics{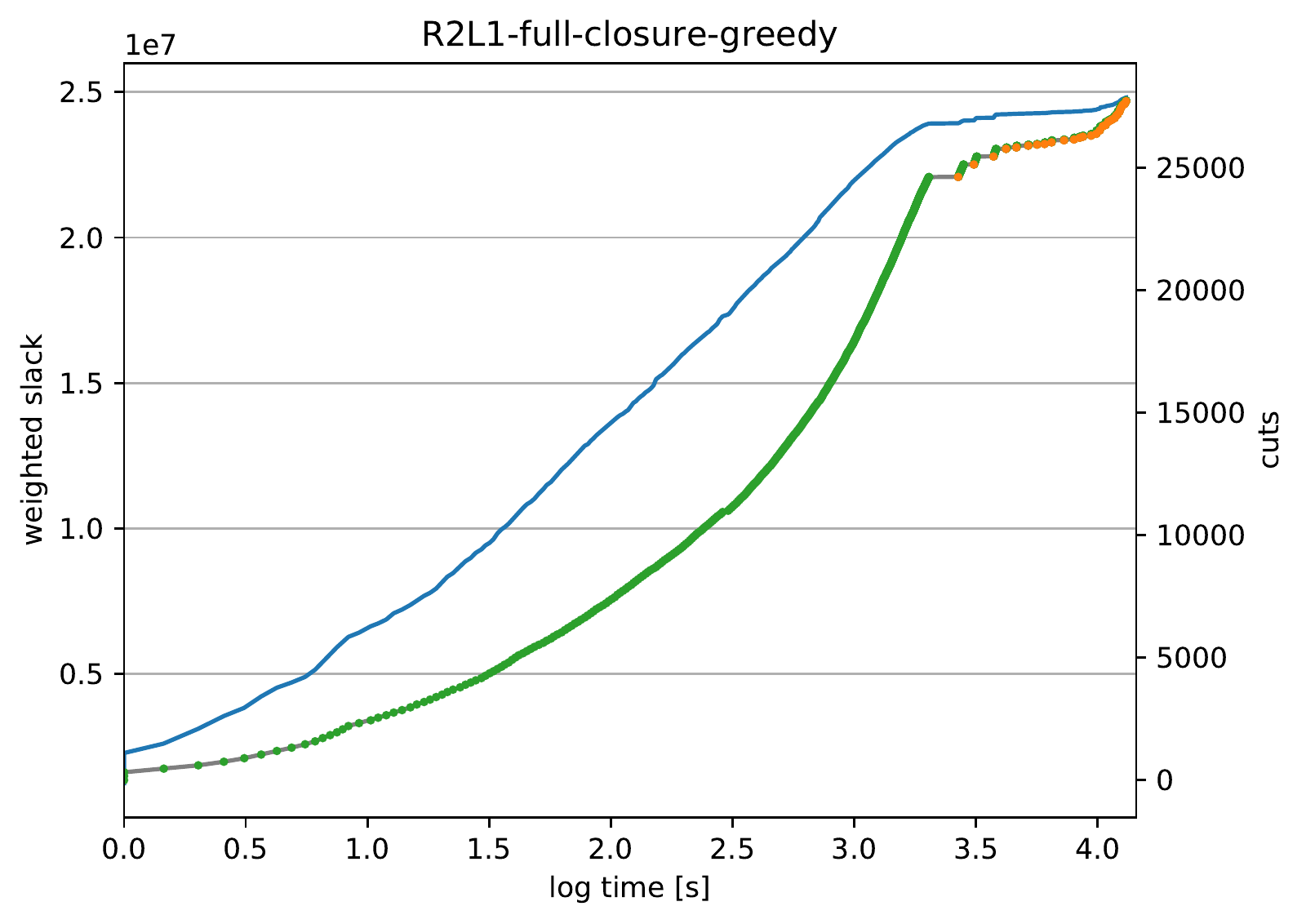}
    \caption{Evolution of the dual bound in terms of weighted slack $w^\top(x-\ell)$ (blue, left axis) and the number of applied split cuts (grey, right axis) for the instance R2L1. Green markers correspond to cuts obtained from the heuristic, orange to cuts from the parametric IP. The time axis is logarithmic.}
    \label{fig:R2L1-full-progress}
\end{figure}

\begin{table}[htbp]
\begin{center}
\begin{tabular}{lrrrrrrr}
Instance & $\mu$ & Dual Bd.\ ($\mathcal P_\mathrm{I}$) & Dual Bd.\ ($\mathcal P_\mathrm{split}$) & Gap [\%] & Dual Bd.\ Source\\
\hline
BL1   & 5\,298 & 3\,668\,148  &  4\,252\,778 & $-15.94$ & \cite{borndorfer_concurrent_2020} \\
BL2   & 4\,880 & 3\,943\,811  &  4\,299\,517 & $-9.02$ & \cite{borndorfer_concurrent_2020} \\
BL3   & 6\,265 & 3\,571\,976  &  4\,290\,946 & $-20.13$ & \cite{borndorfer_concurrent_2020} \\
BL4   & 9\,684 & 3\,131\,491  &  3\,923\,974 & $-25.31$  & \cite{borndorfer_concurrent_2020} \\
\hline
R1L1  & 2\,722 & 20\,901\,883  & 19\,041\,890 & 8.90 & \cite{lindner_forward_2021} \\
R1L2  & 2\,876 & 19\,886\,799  & 19\,059\,669 & 4.16 & \cite{masing_forward_2023} \\
R1L3  & 2\,848 & 19\,323\,821  & 18\,193\,974 & 5.85 & \cite{masing_forward_2023} \\
R1L4  & 3\,769 & 17\,283\,850  & 16\,441\,121 & 4.88 & \cite{masing_forward_2023}\\
\hline
R2L1  & 3\,206 & 25\,929\,643  & 24\,806\,675 & 4.33 & \cite{masing_forward_2023}\\
R2L2  & 3\,360 & 25\,642\,692  & 24\,464\,467 & 4.59 & \cite{masing_forward_2023}\\
R2L3  & 3\,239 & 23\,941\,492  & 22\,645\,939 & 5.41 & \cite{masing_forward_2023}\\
R2L4  & 5\,514 & 19\,793\,447  & 19\,102\,410 & 3.49 & \cite{masing_forward_2023}\\
\hline
R3L1  & 4\,630 & 26\,825\,864  & 25\,343\,534 & 5.53 & \cite{masing_forward_2023}\\
R3L2  & 4\,800 & 27\,178\,406  & 25\,963\,773 & 4.47 & \cite{masing_forward_2023}\\
R3L3  & 5\,446 & 23\,007\,043  & 22\,273\,090 & 3.19 & \cite{masing_forward_2023}\\
R3L4  & 7\,478 & 17\,432\,725  & 17\,027\,192 & 2.33 & \cite{masing_forward_2023}\\
\hline
R4L1  & 5\,331 & 29\,174\,444  & 27\,938\,824 & 4.24 & \cite{masing_forward_2023}\\
R4L2  & 5\,688 & 28\,664\,399  & 27\,585\,028 & 3.77 & \cite{masing_forward_2023}\\
R4L3  & 6\,871 & 24\,293\,621  & 23\,849\,465 & 1.83 & \cite{masing_forward_2023} \\
R4L4  & 9\,371 & 17\,961\,400  & 16\,488\,684 & 8.20 & \cite{lindner_determining_2020} \\
\hline
R1L1v & 2\,832 & 29\,620\,775 & 28\,544\,123 & 3.63 & \cite{PESPlib}\\
R4L4v & 9\,637 & 32\,296\,041 & 38\,307\,814 & $-18.61$ & \cite{PESPlib}\\
\end{tabular}
\caption{Comparison of dual bounds for the full PESPlib instances. The table lists the best known dual bound for the MIP \eqref{eq:mip-cycle} in terms of weighted slack $w^\top(x- \ell)$ according to the source in the last column, the best dual bound obtained by split cuts, and the primal-dual gap.}
\label{tab:results-full-dual}
\end{center}
\end{table}


\subsection{Insights}

From our experiments we have gained two main insights: 
On one hand, we have seen that our procedure is indeed useful in computing qualitative dual bounds, as we were able to improve five instances of the benchmarking library PESPlib significantly. But also for the other instances, some of which have been treated excessively in the past, a high percentage of the bound could be reached in comparably little time by our procedure. 

On the other hand, our tests help us to assess the quality of the split closure for computing the lower bounds independent of the procedure chosen: The instances where the optimal solution could be obtained and our procedure terminates give an indication of how well suited the split closure is for dual bounds in the context of PESP. Here, we were able to observe that the split closure provided fairly low optimality gaps on average, and even certified optimality in three cases, a non-negligible gap remains: E.g., in the worst case, namely for R1L2 with $\mu = 100$, there is a gap of 18.9\, \% between the optimal dual bound of the split closure and the optimal solution. We reach the conclusion that the split closure is essential for raising the dual bound. However, in order to close the primal-gap entirely, further methods, e.g., higher rank split cuts, will have to be applied.

Considering that usually $\mathcal P_{I} \subsetneq \mathcal P_{\mathrm{split}}$, such that any bound obtained from the split closure will not be sufficient to prove optimality, one could ask the question, whether it is worth it to explore it to its full extent, or whether the fast, heuristic section of our procedure would be sufficient. For an indication, we analyzed the instances where our procedure terminated before the time limit was reached:
We found that the best bound before the parametric IP came into effect reached at least 89.1\,\%, and on average even 95.6\,\% of the final dual bound. 
We conclude that indeed the heuristic approach of separating flip inequalities is quite effective, as it is able to cover the majority of the dual bound that can be obtained from the split closure quickly. In our case, the addition of the parametric IP in the procedure was essential for the assessment of the split closure and might be helpful to find new cuts, so that the heuristic can produce effective cuts again. However for practical purposes, particularly when other methods aimed at improving the dual bounds are used in parallel, the time-consuming parametric IP might be too costly. The heuristic part could be sufficient, particularly in light of the realization that -- also in practice -- the split closure is not enough to close the dual gap entirely.


\section{Conclusion}
\label{sec:conclusion}

We have shown that in the context of periodic timetabling, the split closure can be expressed in combinatorial terms, namely via flip inequalities with respect to simple cycles. Consequently, this means that a dual bound obtained from flip inequalities is as good as from split cuts. However, flip inequalities are -- in a way -- easier to grasp: We show that for a fixed cycle, a separating flip inequality can be found in linear time. This can be used to obtain a heuristic, which turned out to be powerful in practice. In combination with a systematic exploration of violated flip inequalities, we were able to improve the dual bounds of five instances of the benchmark library PESPlib -- proving both the effectiveness of our approach, but also of the benefit of the split closure in the context of PESP. 
One of our main contributions is also in the insight that the split closures of various equivalent PESP formulations are all equivalent as well, meaning that neither the specific MIP formulation, nor any amount of subdivision or augmentation will lead to a stronger split closure. 

Our computational experiments also indicate that even with a full exploration of the flip polytope, a certain gap will remain. To close the primal-dual gap entirely, further research into stronger cuts is needed, which will have to be different from first-order split cuts. 
\bibliography{references}

\end{document}